\documentclass[12pt]{article}
\usepackage{amsmath}
\usepackage{amsfonts}
\usepackage{amsthm}
\usepackage{amssymb}
\usepackage{graphicx,psfrag,epsf}
\usepackage{enumerate}
\usepackage{natbib}
\usepackage{url} % not crucial - just used below for the URL 

\usepackage{tikz}
\usetikzlibrary{matrix}
\usetikzlibrary{calc}
\usetikzlibrary{decorations.pathreplacing}

\usepackage[boxed, ruled, vlined]{algorithm2e}

\newcommand{\blind}{0}
\newcommand\numberthis{\addtocounter{equation}{1}\tag{\theequation}}

\newtheorem{theorem}{Theorem}[section]
\newtheorem{definition}{Definition}[section]
\newtheorem{corollary}{Corollary}[section]
\newtheorem{remark}{Remark}[section]
\newtheorem{assumption}{Assumption}

\newtheorem{lemma}{Lemma}[section]
\newtheorem{aplemma}{Lemma}[section]

\begin{document}

%%%%%%%%%%%%%%%%%%%%%%%%%%%%%%%%%%%%%%%%%%%%%%%%%%%%%%%%%%%%%%%%%%%%%%%%%%%%%%

\if0\blind
{
  \title{\bf Projection Pursuit for non-Gaussian Independent Components \thanks{Joni Virta is PhD student, University of Turku, 20014, Finland (E-mail: \textit{joni.virta@utu.fi}). Klaus Nordhausen is Postdoctoral Researcher, University of Turku, 20014, Finland (E-mail: \textit{klaus.nordhausen@utu.fi}). Hannu Oja is Professor Emeritus, University of Turku, 20014, Finland (E-mail: \textit{hannu.oja@utu.fi}). This work was supported by Academy of Finland (grant 268703). The authors wish to thank Jari Miettinen who provided the code used for the comparison of the individual signal estimates in Section \ref{sec:simu}.}\ \thanks{This work is partially based on the unpublished manuscript \cite{virta2015joint} available as an arXiv preprint, arXiv:1505.02613.}}
  \author{Joni Virta, Klaus Nordhausen and Hannu Oja \hspace{.2cm}\\
    Department of Mathematics and Statistics,\\University of Turku, 20014, Finland}
  \maketitle
} \fi

\if1\blind
{
  \bigskip
  \bigskip
  \begin{center}
    {\LARGE\bf Projection Pursuit for non-Gaussian Independent Components}
\end{center}
  \medskip
} \fi

\begin{abstract}
In independent component analysis it is assumed that the observed random variables are linear combinations of latent, mutually independent random variables called the independent components. Our model further assumes that only the non-Gaussian independent components are of interest, the Gaussian components being treated as noise. In this paper projection pursuit is used to extract the non-Gaussian components and to separate the corresponding signal and noise subspaces. Our choice for the projection index is a convex combination of squared third and fourth cumulants and we estimate the non-Gaussian components either one-by-one (deflation-based approach) or simultaneously (symmetric approach). The properties of both estimates are considered in detail through the corresponding optimization problems, estimating equations, algorithms and asymptotic properties. Various comparisons of the estimates show that the two approaches separate the signal and noise subspaces equally well but the symmetric one is generally better in extracting the individual non-Gaussian components.
\end{abstract}

\noindent%
\vfill

\newpage

\section{INTRODUCTION}

\subsection{Blind source separation model and its extensions}

The basic \textit{blind source separation} (BSS) model assumes that the observed random vectors $\textbf{x}_i \in \mathbb{R}^p$ are linear combinations of some unobservable random vectors $\textbf{z}_i \in \mathbb{R}^p$, $i=1,\ldots,n$, the estimation of which is the main objective. This can be formalized as
\begin{equation}\label{eq:bss_model}
\textbf{x}_i = \boldsymbol{\mu} + \boldsymbol{\Omega} \textbf{z}_i, \quad i=1,\ldots,n,
\end{equation}
where $\boldsymbol{\mu}$ is a location shift and $\boldsymbol{\Omega} \in \mathbb{R}^{p \times p}$ is a non-singular \textit{mixing matrix}.
In \textit{independent component analysis} (ICA) it is further assumed that the random vector $\textbf{z}_i$ has mutually independent and standardized components and that at most one of the components  is normally distributed. The constraint on non-Gaussianity is needed as otherwise the rotational invariance of the standard multivariate Gaussian distribution makes the model ill-defined. For an overview of ICA and BSS models, see e.g. \cite{comon2010handbook}.

However, if it is reasonable to assume the existence of more than one source of noise it might be too strict to restrict the number of Gaussian components to at most one. In the so-called \textit{non-Gaussian component/subspace analysis (NGCA)} \citep{blanchard2005non} the assumptions in the model \eqref{eq:bss_model} are relaxed by allowing more sources of Gaussian noise. That is, as formulated in \cite{theis2011uniqueness}, one assumes that for some $d$, $0\le d\le p$ the vector $\textbf{z}_i$ consists of a $d$-dimensional non-Gaussian subvector (that spans the \textit{signal space}) and a $(p-d)$-dimensional Gaussian subvector (that spans the \textit{noise space}) that are independent of each other. %It then follows that, in this model,
 %the $p-d$  columns of $\boldsymbol{\Omega}$ corresponding to the Gaussian components are not identifiable. However, this is not a problem as
In this model the signal and the noise subspaces are well defined and estimable but the individual components are not. A further distinctive property of NGCA is that the dimension of the signal subspace is usually assumed to be known. For other related or similar models, see \citet{BonhommeRobin2009}, \citet{comon2010handbook}, \citet{WoodsHansenStrother2015} and \citet{RiskMattesonRuppert2015}.

In this paper we combine the ICA and NGCA approaches and require that in \eqref{eq:bss_model} exactly $p-d$ of the mutually independent components of $\textbf{z}_i$ are Gaussian. The objective is then to estimate the signal and noise subspaces (as in NGCA) as well as the individual signal components (as in ICA). In the estimation we use \textit{projection pursuit} (PP) with convex combinations of squared third and fourth cumulants as projection indices. We derive the properties of the estimates assuming that $d$ is known, but also discuss its estimation in Section \ref{sec:sub}.

\subsection{Projection pursuit in ICA}

Projection pursuit is a popular method that finds hidden structures in multivariate data by searching for one- or low-dimensional projections of interest.
This is done by finding  projections
%one or several orthogonal linear combinations of the original variables
that maximize the value of an objective function, the so-called projection index. The classical measures of skewness and kurtosis, the third and fourth moments of a random variable after standardization, have been widely used for this purpose. \cite {Huber:1985} considered various projection indices with heuristic arguments that non-Gaussian linear combinations are the most interesting. His indices were ratios of two dispersion functionals thus measuring kurtosis with the classical kurtosis measure as a special case. \cite{pena2001cluster} used projection pursuit for hidden cluster identification with the classical kurtosis measure. For early contributions on projection pursuit see also \cite{FriedmanTukey:1974} and \cite{jones1987projection}.

In the engineering literature \cite{HyvarinenOja:1997} were the first to propose a projection pursuit approach for independent component analysis with the absolute value of the excess kurtosis, the fourth cumulant of a standardized random variable, as the projection index and considered later an extension with a choice among several alternative measures of non-Gaussianity, including the absolute value of the classical skewness, the third cumulant of a standardized random variable. The approach is called {\it deflation-based  FastICA} or {\it symmetric FastICA} depending on whether the independent components are estimated one-by-one or simultaneously, respectively. These two versions of FastICA can actually be seen as the optimization of different norms of the same vector of component-wise criterion values, namely the repeated maximization of $L_\infty$-norm and the maximization of $L_1$-norm, respectively. In this paper we choose however to work with the $L_2$-norm, firstly because of its analytical tractability over other choices of $L_p$-norms, and secondly because of its connection to the popular JADE estimate, see Corollary \ref{cor:asymp_equiv}. The deflation-based approach also has the useful property that under our model the estimation of each sequential component can be seen as a test for normality, allowing us to do inference on the dimension $d$ of the signal subspace, see Section \ref{sec:sub}. %As far as we know, no definitive work has previously been carried out for cases other than $L_1$ and $L_\infty$.
Although the use of the $L_2$-norm was mentioned already in \cite{hyvarinen1999fast} and \cite{comon2010handbook} the idea was not carried any further. For recent discussions of FastICA methods, see \cite{Ollila2010}, \cite{nordhausen2011deflation},  \cite{MiettinenNordhausenOjaTaskinen:2014},  \cite{miettinen2014fourth}, \citet{Wei2015}. For other classical approaches to ICA, see e.g. \cite{cardoso1989source}, \cite{cardoso1993blind} and the application of the latter to NGCA in \citet{kawanabe2005linear}.

In the engineering literature the ICA procedures are often seen more as numerical algorithms than as estimates of certain population quantities and considering their statistical properties is usually neglected. Recently, however, also more statisticians have become interested in the problem. \citet{ChenBickel:2006} and \citet{SamworthYuan:2012}, for example, developed estimates that need only the existence of first moments and rely on efficient nonparametric estimates of the marginal densities. Efficient estimation methods based on residual signed ranks and residual ranks have also been developed recently by \citet{IlmonenPaindaveine:2011} and \citet{HallinMehta:2015}. For other approaches see also, for example, \citet{KarvanenKoivunen:2002, HastieTibshirani2003, matteson2016independent}.

% In NGCA normally distributed components are treated as noise and  the main concern is the estimation of the \textit{signal subspace}, that is, the subspace of the non-Gaussian components. See, for example, \cite{blanchard2005non} for a simultaneous use of several indices and \cite{kawanabe2005linear} for a modification of JADE for the NGCA model.

\subsection{The joint use of multiple cumulants}

Regarding signal separation, one limitation to the previously discussed methods is set by the measure of non-Gaussianity itself; if some signal component has a distribution with the same criterion value as the normal distribution (usually zero) it is treated as noise. This is where the benefits of our preferred approach show the most; jointly using both third and fourth cumulants means that the signal components can have either zero skewness or zero excess kurtosis and still be estimated. Furthermore, we show that the method is Fisher consistent under suitable moment conditions. Theoretically any number of univariate cumulants of order three or higher could be used together to find the non-Gaussian signals, see \citet{moreau2001generalization}.

The joint use of third and fourth cumulants in ICA-type problems has also been previously discussed in the literature: \cite{jones1987projection} approximate the entropy of a distribution suitably close to the Gaussian distribution with a particular weighted sum of squared skewness and excess kurtosis, a special case of the projection index we use, see also \cite{amari1996new} for similar constructions; \cite{KarvanenKoivunen:2002, karvanen2002adaptive} use the method of moments to estimate the source score functions either from Pearson's system of distributions or the extended generalized lambda distribution (EGLD) by matching moments (or L-moments) up to the fourth; \cite{karvanen2004independent} propose using a weighted sum of the absolute values of skewness and excess kurtosis as an objective function in minimization of mutual information; \cite{comon2006blind} used the characteristic function to solve an independent component problem with more latent variables than observed variables; \cite{li2011joint} proposed using a joint diagonalization of second or higher order cumulant matrices in the context of joint blind source separation (JBSS) and \cite{comon2015polynomial} considered the simultaneous decomposition of multiple symmetric tensors of different orders.

\subsection{The structure of the paper}
The paper is structured as follows; we begin in Section \ref{sec:nota} by providing some notation. Section \ref{sec:mod_ssf} introduces the model along with the relevant assumptions and also discusses the concepts of signal separation functionals and affine equivariance. Sections \ref{sec:sep} and \ref{sec:sym} then consider the estimation of the signals both separately and simultaneously via projection pursuit with the previously unexplored use of the $L_2$-norm. A thorough discussion is provided including also the asymptotic behaviors of both considered methods. In Sections \ref{sec:simu} and \ref{sec:sub} the presented procedures are compared in their ability to extract single signals and the entire signal subspace. In the latter some thought on estimating the value of $d$ is also given. We end with some discussion on the results and prospective work in Section \ref{sec:conc}. The proofs are reserved for the Appendix. %\cite{virta2016supplementary}.

The simulations and computations were performed using R 3.2.3 \citep{R} and additionally the packages \textit{clue} \citep{Rclue}, \textit{ggplot2} \citep{Rggplot2}, \textit{ICS} \citep{Rics}, \textit{JADE} \citep{Rjade}, \textit{Rcpp} \citep{Rrcpp} and \textit{Rcpp\-Armadillo} \citep{Rrcpparmadillo}.

\section{NOTATION} \label{sec:nota}

\subsection{Univariate moments}

For a univariate random variable $x$, we write $x_{st}:=(x-E(x))/\sqrt{Var(x)}$ for its standardized version. The classical skewness, kurtosis and excess kurtosis of $x$ are then
\[
\gamma(x):=E\left(x_{st}^3\right),\ \ \beta(x):=E\left(x_{st}^4\right)\ \ \mbox{and}\ \
\kappa(x):=\beta(x)-3.
\]
Note that the measures $\gamma(x)$ and $\kappa(x)$ are the third and fourth cumulants of the standardized variable $x_{st}$. For symmetric random variables $\gamma(x)=0$ and
for the normal distribution $\kappa(x)=0$.

Throughout the paper we assume that $ \textbf{z}_1,\ldots,\textbf{z}_n$ is a random sample from a $p$-variate distribution of $\textbf{z}$  with $E(\textbf{z})= \mathbf{0}$ and $Cov(\textbf{z})= \textbf{I}_p$ and that the $p$ components of $\textbf{z}$ are mutually independent. As the methods considered are essentially moment-based we use, for all $k=1,\ldots,p$, the following shorthands for marginal moments.
\begin{align*}
\gamma_k &:= E(z_{ik}^3), \quad &\beta_k &:= E(z_{ik}^4), \quad  &\kappa_k &:= E(z_{ik}^4) - 3, \\
\nu_k &:= E(z_{ik}^4) - 1, \quad &\omega_k &:= E(z_{ik}^6) - E(z_{ik}^3)^2, \quad &\eta_k &:= E(z_{ik}^5) - E(z_{ik}^3).
\end{align*}

Assuming a fixed weight parameter $\alpha \in [0, 1]$, we will see in Sections \ref{sec:sep} and \ref{sec:sym} that the previous moments act as building blocks for the following asymptotic variances of the elements of our unmixing matrix estimates.
\begin{align*}
A_k &:= \frac{\alpha_1^2 \zeta^{(3)}_{k} + 2 \alpha_1 \alpha_2 \zeta^{(34)}_{k} + \alpha_2^2 \zeta^{(4)}_{k}}{(\alpha_1 \gamma_k^2 + \alpha_2 \kappa_k^2)^2}, \\
B_{kl} &:= \frac{\alpha_1^2 (\zeta^{(3)}_k + \zeta^{(3)}_l + \gamma_l^4) + 2 \alpha_1 \alpha_2 (\zeta^{(34)}_k + \zeta^{(34)}_l + \gamma_l^2\kappa_l^2) + \alpha_2^2 (\zeta^{(4)}_k + \zeta^{(4)}_l + \kappa_l^4)}{(\alpha_1 (\gamma_k^2 + \gamma_l^2) + \alpha_2 (\kappa_k^2 + \kappa_l^2))^2}, \\
D_k &:= \frac{\kappa_k + 2}{4},
\end{align*}
where $\alpha_1 := 3\alpha$, $\alpha_2 := 4(1- \alpha)$, $\zeta^{(3)}_{k} = \gamma_k^2 (\nu_k - \gamma_k^2)$, $\zeta^{(4)}_{k} = \kappa_k^2 (\omega_k - \beta_k^2)$ and $\zeta^{(34)}_{k} = \gamma_k \kappa_k (\eta_k - \gamma_k \beta_k)$. The expressions $A_k$, $B_{kl}$ and $D_k$ are encountered in Theorems \ref{theo:sep_asymp} and \ref{theo:sym_asymp}.

\subsection{Vector- and matrix-valued quantities}

We write $F_\textbf{x}$ for the cumulative distribution function (c.d.f.) of a $p$-variate random vector $\textbf{x}$. Then  $F_{\textbf{A} \textbf{x} + \textbf{b}}$ is the c.d.f. of  $\textbf{A} \textbf{x} + \textbf{b}$. If the random vector $\textbf{x}$ has the mean vector $\boldsymbol{\mu}$ and the covariance matrix $\boldsymbol{\Sigma}$, the standardized vector $\textbf{x}_{st}$ is given by $\textbf{x}_{st} := \boldsymbol{\Sigma}^{-1/2} (\textbf{x} - \boldsymbol{\mu})$, where $\boldsymbol{\Sigma}^{-1/2}$ is chosen as the unique symmetric matrix $\textbf{G}$ satisfying $\textbf{G} \boldsymbol{\Sigma} \textbf{G}^T = \textbf{I}_p$. Multivariate standardization is affine equivariant up to rotation, that is, if $\textbf{x}^* = \textbf{A} \textbf{x} + \textbf{b}$, then $(\textbf{x}^*)_{st} = \textbf{V} \textbf{x}_{st}$, for some orthogonal matrix $\textbf{V} \in \mathbb{R}^{p \times p}$, see e.g. \cite{ilmonen2012invariant}. This result is crucial in proving the affine equivariance of the signal separation functionals later on.

Fixing $\alpha \in [0, 1]$, the projection pursuit methods in the following sections are based on the objective function
\begin{align}\label{eq:index_defl}
G_\alpha(\textbf{u}) := \alpha \gamma^2(\textbf{u}^T \textbf{x}_{st}) + (1 - \alpha) \kappa^2(\textbf{u}^T \textbf{x}_{st}),
\end{align}
where $\textbf{x}$ is the observed random vector and $\textbf{u}^T\textbf{u} = 1$. For the estimating equations we further need its gradient,
\begin{align}
\textbf{T}_\alpha(\textbf{u}) = 3 \alpha \gamma( \textbf{u}^T \textbf{x}_{st}) E\left[(\textbf{u}^T \textbf{x}_{st})^2 \textbf{x}_{st} \right] + 4 (1 - \alpha) \kappa( \textbf{u}^T \textbf{x}_{st})  E\left[(\textbf{u}^T \textbf{x}_{st})^3 \textbf{x}_{st} \right].
\end{align}\label{eq:grad_defl}
\begin{remark}\label{rem:grad}
Later in this paper the gradient $\textbf{T}_\alpha(\textbf{u})$ is used to build a fixed-point algorithm for the orthonormal rows of the matrix $\textbf{U}$, see Algorithms \ref{algo:sep} and \ref{algo:sym1}.
In practice more stable algorithms are obtained by replacing the gradient with the following alternative based on a modified Newton-Raphson algorithm.
\begin{align*}
\textbf{T}{}^*_\alpha(\textbf{u}) &= 3 \alpha \gamma( \textbf{u}^T \textbf{x}_{st}) E\left[(\textbf{u}^T \textbf{x}_{st})^2 \textbf{x}_{st} \right] \\
&+ 4 (1 - \alpha) \kappa( \textbf{u}^T \textbf{x}_{st}) \left(  E\left[(\textbf{u}^T \textbf{x}_{st})^3 \textbf{x}_{st} \right] - 3 \textbf{u} \right),
\end{align*}
see e.g. \citet{HyvarinenOja:1997} and \citet{Miettinen2016} for further discussion.
\end{remark}
%$(\partial/\partial \textbf{u})\textbf{G}_\alpha(\textbf{u}) =: \textbf{T}_\alpha(\textbf{u})$.
%The algorithms given later on apply for any choice of [SHOULD IT BE TWICE] (differentiable) $G$, but in Sections \ref{sec:sep} and \ref{sec:sym} we prove that the above choice actually succeeds in the task of estimating the signals. For some other choice this might not hold. IS THIS REALLY NEEDED?
%Some thought on choosing an appropriate weighting is given in Section \ref{sec:sub}.
%The standard basis vectors of $\mathbb{R}^p$ are denoted by $\textbf{e}_i := (\delta_{1i},...,\delta_{pi})$, where $\delta_{ij}$ is Kronecker's delta. Using the standard basis vectors we define the following matrices
%\[\textbf{E}^{ij} = \textbf{e}_i \textbf{e}_j^T, \qquad i,j=1,...,p,\]
%the only non-zero element of $\textbf{E}^{ij}$ being the element $(i, j)$. Denote further by $\textbf{I}_{d, p} \in \mathbb{R}^{d \times p}$ the first $d \leq p$ rows of the $p$-dimensional identity matrix.
Some often encountered sets of square matrices include the set of all orthogonal matrices, the set of all heterogeneous sign-change matrices (diagonal matrices with diagonal elements equal to $\pm 1$), the set of all heterogeneous scaling matrices (diagonal matrices with positive diagonal elements) and the set of all permutation matrices denoted, respectively, by $\mathcal{U}, \mathcal{J}, \mathcal{D}$ and $\mathcal{P}$.
The set of all $d \times p$ matrices with orthonormal rows is denoted by $\mathcal{U}^{d \times p}$ and the equivalence relation $\textbf{A} \equiv \textbf{B}$ means that $\textbf{A} = \textbf{P} \textbf{J} \textbf{B}$, for some $\textbf{P} \in \mathcal{P}$, $\textbf{J} \in \mathcal{J}$. Also, let $\mathcal{C}$ denote the set of all matrices that can be expressed as $\textbf{J} \textbf{D} \textbf{P}$ where $\textbf{J} \in \mathcal{J}$, $\textbf{D} \in \mathcal{D}$ and $\textbf{P} \in \mathcal{P}$.
%\begin{itemize}
%\item $\mathcal{U} = \{\textbf{U} \in \mathbb{R}^{p \times p} \,: \, \textbf{U} \text{ is an orthogonal matrix.}\}$
%\item $\mathcal{J} = \{\textbf{J} \in \mathbb{R}^{p \times p}\,: \, \textbf{J} = diag(j_1,...,j_p), \, j_1,...,j_p = \pm 1\}$
%\item $\mathcal{D} = \{\textbf{D} \in \mathbb{R}^{p \times p}\,: \, \textbf{D} = diag(d_1,...,d_p), \, d_1,...,d_p > 0\}$
%\item $\mathcal{P} = \{\textbf{P} \in \mathbb{R}^{p \times p}\,: \, \textbf{P} \text{ is a permutation matrix.}\}$
%\end{itemize}

%Finally, the absolute value $| \cdot |$ in Algorithms \ref{algo:sep}, \ref{algo:sym1} and \ref{algo:sym2} is  understood to be taken component-wise and
The norm $\| \cdot \|$ refers for vector arguments to the standard Euclidean norm and for matrix arguments to the Frobenius norm.

\section{SIGNAL SEPARATION MODEL}\label{sec:mod_ssf}

\subsection{The model and its assumptions}\label{subsec:ICmodel}

The model used throughout the paper is the following combination of the ICA and NGCA models in which the $p$-variate observations $\textbf{x}_1,\ldots, \textbf{x}_n$ are thought to be independent realisations of the random vector $\textbf{x} \in \mathbb{R}^p$ generated as
\begin{align}\label{eq:icm_icmodel}
\textbf{x} = \boldsymbol{\mu} + \boldsymbol{\Omega} \textbf{z}, \quad \mbox{where } \textbf{z} = \begin{pmatrix}
\textbf{s} \\
\textbf{n}
\end{pmatrix}, \,
\textbf{s} \in \mathbb{R}^{d}, \, \textbf{n} \in \mathbb{R}^{p - d}.
\end{align}
We further assume that the unobserved random vector $\textbf{z}$ satisfies the following two assumptions.

\begin{assumption}\label{assu:ind}
The components of $\textbf{z}$ are mutually independent and standardized.
\end{assumption}

\begin{assumption}\label{assu:gaus}
The components of $\textbf{s}$ are non-Gaussian and the components of $\textbf{n}$ are Gaussian.
\end{assumption}
%(Note that in the following also non-subscripted $\textbf{x}$ and $\textbf{z}$ are used when not referring to any specific observation.)
The standardization conditions $E(z_k)=0$ and $E(z_k^2)=1$, $k=1,\ldots,p$, implied by Assumption \ref{assu:ind} serve as identification constraints for the location $\boldsymbol{\mu}$ and the scales of the columns of $\mathbf{\Omega}$, implying that  $E(\textbf{x})=\boldsymbol{\mu}\ \ \mbox{and}\ \ Cov(\textbf{x})=\boldsymbol{\Sigma} = \boldsymbol{\Omega} \boldsymbol{\Omega}^T$.
Writing $\boldsymbol{\Omega} = (\boldsymbol{\Omega}_1, \boldsymbol{\Omega}_2)$ the model can also be expressed as
\begin{align*}%\label{eq:icm icmodel2}
 \textbf{x} = \boldsymbol{\mu} + \boldsymbol{\Omega}_1 \textbf{s} + \boldsymbol{\Omega}_2 \textbf{n},
\end{align*}
where $\boldsymbol{\Omega}_2\in \mathbb{R}^{p\times (p-d)}$ can be identified only up to a rotation from the right-hand side. The matrix  $\boldsymbol{\Omega}_1\in \mathbb{R}^{p\times d}$,
however, can be identified up to the signs and permutation of its columns, a fact which follows from the classical Skitovich-Darmois theorem restated here as Theorem \ref{theo:ibragimov}, see Theorem 1 in \citet{ibragimov2014ghurye}. The proof of Corollary \ref{cor:identify} is given in the Appendix.
\begin{theorem}\label{theo:ibragimov}
Assume that $s_1, \ldots, s_k$ are independent random variables  and that the real numbers $a_1,\ldots,a_k,b_1,\ldots,b_k$ are non-zero. If $\sum_{j=1}^k a_j s_j$ and $\sum_{j=1}^k b_j s_j$ are independent then the variables $s_1,\ldots,s_k$ are normally distributed.
\end{theorem}
\begin{corollary}\label{cor:identify}
The submatrix $\boldsymbol{\Omega}_1$ can be identified up to the signs and permutation of its columns.
\end{corollary}

Assumption \ref{assu:gaus} quantifies our prior knowledge on the dimensions of the non-Gaussian signal subspace and the Gaussian noise subspace. Inference, testing and estimation of the dimension of the signal subspace is also briefly discussed in Section \ref{sec:sub}. In addition to Assumptions \ref{assu:ind} and \ref{assu:gaus}, for the limiting distributions of our unmixing matrix estimates to exist the existence of some higher moments of $\textbf{z}$ is further required, see Theorems \ref{theo:sep_asymp} and \ref{theo:sym_asymp}.

In the case of $d=0$ we simply get a multivariate normal model without any signal $\textbf{s}$ and if $d=p-1$ or $d=p$ one gets the basic independent component model. Another related model is the so-called {\it noisy independent component model}, $\textbf{x} = \mathbf{\Omega} \textbf {s} + \textbf {n}$, with $d$ non-Gaussian independent components in $\textbf{s}$ and Gaussian noise in the $p$-variate $\textbf{n}$ with $\mathbf{\Omega} \in \mathbb{R}^{p\times d}$. Stacking the signal and noise parts to a single vector, the effective mixing matrix $( \boldsymbol{\Omega} | \textbf{I})$ is then non-square and standard ICA methods no longer apply. For such models, see e.g. \citet{BonhommeRobin2009} and \citet{comon2010handbook}.

\subsection{Signal separation functionals}\label{subsec:ssf}

Our approach for estimating the signal components and the signal subspace is projection pursuit with a preliminary step advised by the following theorem. The theorem is one of the key results of independent component analysis, see for example \cite{miettinen2014fourth} for its proof.

\begin{theorem}\label{theo:icm_rotate}
Let $\textbf{x} \in \mathbb{R}^p$ follow the signal separation model in \eqref{eq:icm_icmodel}. Then the standardized vector $\textbf{x}_{st} = \boldsymbol{\Sigma}^{-1/2} (\textbf{x} - \boldsymbol{\mu})$
satisfies $\textbf{s} = \textbf{U} \textbf{x}_{st}$ for some $\textbf{U} \in \mathcal{U}^{d \times p}$.
\end{theorem}
Theorem \ref{theo:icm_rotate} essentially states that the estimation of the relevant rows of $\boldsymbol{\Omega}^{-1}$ can, using standardization, in fact be reduced to a simpler problem, namely to the estimation of a matrix $\textbf{U}$ with orthonormal rows, and this is the task we use projection pursuit in the next sections for. Having estimated $\textbf{U}$ a transformation into the signal space is then given by $\textbf{x} \mapsto \textbf{W} \textbf{x}$, where $\textbf{W} := \textbf{U} \boldsymbol{\Sigma}^{-1/2}$. %Partitioning the transformation matrix as $\textbf{W} = (\textbf{W}_1 | \textbf{W}_2)$, where $\textbf{W}_1 \in \mathbb{R}^{d \times d}$ and $\textbf{W}_2 \in \mathbb{R}^{d \times (p - d)}$, yields further an interesting interpretation in the case $\boldsymbol{\Omega} = \textbf{I}_p$; $\textbf{W}_1$ describes how well the individual signals are separated from each other and $\textbf{W}_2$ measures how well the signals are separated from the noise, see Sections \ref{sec:simu} and \ref{sec:sub} for more discussion. Moreover, if we are not interested in the individual signals but instead only on the subspace spanned by them, an orthogonal projection onto the signal subspace is given by $\textbf{P}_W := \textbf{W}^T (\textbf{W} \textbf{W}^T)^{-1} \textbf{W}$. This approach is discussed in Section \ref{sec:sub}.

Next, we give the definition of a \textit{signal separation functional} which formalizes the role of the transformation matrix $\textbf{W}$ above.

\begin{definition}\label{def:ssf}
The matrix-valued functional $\textbf{W} \in \mathbb{R}^{d \times p}$ is said to be a  signal separation functional if (i) under the model \eqref{eq:icm_icmodel} it holds that $\textbf{W}(F_\textbf{x}) \textbf{x} \equiv \textbf{s}$ and (ii) $\textbf{W}(F_\textbf{x})$ is affine equivariant in the sense that for all $\textbf{x} $ and all full-rank $\textbf{A} \in \mathbb{R}^{p \times p}$,
\[\textbf{W}(F_{\textbf{Ax}}) \textbf{A} \textbf{x} \equiv \textbf{W}(F_\textbf{x}) \textbf{x}.\]
\end{definition}

%The requirement $(i)$ in the definition above can actually be replaced by a weaker one stating that $\textbf{W}(F_\textbf{x}) \textbf{x}$ has independent

%\begin{lemma}
%a
%\end{lemma}

%Since no Gaussian signal components are allowed in the model \eqref{eq:icm_icmodel} the estimated vector of independent components $\textbf{W}_1(F_\textbf{x}) \textbf{x}$ has to be equal to $\textbf{z}_1$ up to sign and order, see the Ghurye-Olkin-Zinger characterization theorem in \citet{ibragimov2014ghurye}. Thus all signal separation functionals $\textbf{W}_1(F_\textbf{x})$ lead to the same unique independent components up to sign change and permutation. This fact combined with Definition \ref{def:ssf} then implies that signal separation functionals are Fisher consistent to the first $d$ rows of $\boldsymbol{\Omega}^{-1}$ up to permutation and heterogeneous sign-change.

The condition $(i)$ in the definition above states that signal separation functionals are up to sign and permutation Fisher consistent to the $d$ rows of $\boldsymbol{\Omega}^{-1}$ corresponding to the signals. Note that this invariance with respect to signs and order is unavoidable since neither is fixed in the model \eqref{eq:icm_icmodel}. The condition $(ii)$ implies in particular that if we consider $F_n$, the empirical distribution function of a random sample $\textbf{x}_1,\ldots,\textbf{x}_n$ from $F$, the estimate $\textbf{W}(F_n)$ is also affine equivariant. This gives us the practical advantage of needing to consider only the case of identity mixing, $\boldsymbol{\Omega} = \textbf{I}_p$, when discussing the estimates' asymptotic behaviors, the case of general $\boldsymbol{\Omega}$ easily following.

The requirement of affine equivariance further means that in the case of $d=p$ any signal separation functional $\textbf{W}(F)$ is also an {\it ICS functional}, that is, $\textbf{W}(F)$ provides a transformation to an {\it invariant coordinate system (ICS)}, see \cite{TylerCritchleyDumbgenOja:2009} and \cite{ilmonen2012invariant}.

\section{ESTIMATING THE SIGNALS SEPARATELY} \label{sec:sep}

Our first approach uses projection pursuit to estimate a single signal at a time, continuing until all $d$ signals have been extracted. The objective function in Definition \ref{def:sep_func} has been discussed already in \cite{jones1987projection} for $\alpha = 0.8$, but only the cases $\alpha = 0$ and $\alpha = 1$ have been previously given a comprehensive treatise in literature, including asymptotics. First, to actually guarantee the validity of our approach, we present the following inequality, an extension of the first part of Theorem 2 in \cite{miettinen2014fourth}.

\begin{theorem}\label{theo:sep_ineq}
Assume the model in \eqref{eq:icm_icmodel} and let $\textbf{U} = (\textbf{u}_1,\ldots,\textbf{u}_{d})^T$ be the matrix of Theorem \ref{theo:icm_rotate}. Without loss of generality, assume further that, for a chosen $\alpha\in [0,1]$, the signals in $\textbf{s}$ are ordered decreasingly according to the values $\alpha \gamma^2 + (1 - \alpha) \kappa^2$. Then for any fixed $k \geq 1$,
%$G_\alpha(\textbf{u}) \leq G_\alpha(\textbf{u}_1)$  for all vectors $\textbf{u} \in \mathbb{R}^p$ satisfying $\textbf{u}^T \textbf{u} = 1$ and
\begin{align*}
G_\alpha(\textbf{u}) \leq G_\alpha(\textbf{u}_k),
\end{align*}
for all $\textbf{u} \in \mathbb{R}^p$ satisfying $\textbf{u}^T \textbf{u} = 1$ and $\textbf{u}^T \textbf{u}_l = 0$ for all $l = 1,\ldots,k-1$.
\end{theorem}
%\alpha \gamma^2(\textbf{u}^T \textbf{z}) + (1 - \alpha) \kappa^2(\textbf{u}^T \textbf{z}) \leq \alpha \gamma_k^2 + (1 - \alpha) \kappa_k^2,
Theorem \ref{theo:sep_ineq} implies that the $k$th row of $\textbf{U}$ is a global maximizer of $G_\alpha$ in the orthogonal complement of $\mbox{span}(\textbf{u}_1,\ldots,\textbf{u}_{k-1})$. Thus the $d$ signals can be recovered by repeatedly searching for mutually orthogonal vectors $\textbf{u}$ maximizing the value of the projection index $G_\alpha$. %Note that the value of $G_\alpha$ is zero in the orthogonal complement of $\mbox{span}(\textbf{u}_1,\ldots,\textbf{u}_{d})$ showing again that the noise vector $\textbf{n}$ can be recovered only up to rotation.
This technique is concretised in the following definition.

\begin{definition}\label{def:sep_func}
For a chosen $\alpha \in [0, 1]$, the deflation-based projection pursuit functional based on convex combination of squared third and fourth cumulants is a functional $\textbf{W} (F_\textbf{x}) = \textbf{U} \boldsymbol{\Sigma}^{-1/2}$, where $\boldsymbol{\Sigma} = Cov(\textbf{x})$ and the orthonormal rows of the matrix $\textbf{U} = (\textbf{u}_1,\ldots,\textbf{u}_d)^T$ are found one-by-one such that
\[ \textbf{u}_k = \underset{\textbf{u}_k^T \textbf{u}_l = \delta_{kl}, 1 \leq l \leq k}{\emph{argmax}} G_\alpha(\textbf{u}_k). \]
\end{definition}

As all $L_p$-norms are equal in $\mathbb{R}$, the proposed method (which uses $L_2$) is equivalent to the first $d$ steps of deflation-based FastICA (which uses $L_1$) \citep{hyvarinen1999fast} whenever using only one of the cumulants. That is, choosing either $\alpha = 0$ or $\alpha = 1$ in Definition \ref{def:sep_func} corresponds to deflation-based FastICA with the respective projection indices $|\kappa(\textbf{u}_k^T \textbf{x}_{st})|$ and $|\gamma(\textbf{u}_k^T \textbf{x}_{st})|$.

Recalling that the transformation $\textbf{x}\to \textbf{A} \textbf{x} + \textbf{b}$ induces the transformation $\textbf{x}_{st}\to  \textbf{V}  \textbf{x}_{st}$ for some orthogonal $\textbf{V}$, the affine equivariance of the procedure given in Definition \ref{def:sep_func} follows simply from the fact that the optimization problem along with its constraints is invariant under mappings $\textbf{x}_{st} \mapsto \textbf{V} \textbf{x}_{st}$, where $ \textbf{V} \in \mathcal{U}$. Thus we have the following result. %This is formalized in the following lemma.

%This together with Theorem \ref{theo:sep_ineq} implies the following.

\begin{lemma}\label{lem:sep_ae}
The deflation-based projection pursuit functional $\textbf{W} (F_\textbf{x})$ in Definition \ref{def:sep_func} is a signal separation functional for every $\alpha \in [0, 1]$.
\end{lemma}

In practice the solution for the $k$th step can be found with the following fixed-point algorithm, the derivation of which is given in the Appendix.%However, a Newton-Raphson type algorithm for this problem is most likely more efficient and will be considered in a separate paper.

\begin{algorithm}[H]
\SetAlgoLined
Initialize $\textbf{u}_{k}$\;
$\Delta \longleftarrow \infty$\;
\While{$ \Delta > \epsilon$}{
$\textbf{u}_{k, new} \longleftarrow \left(\textbf{I}_p - \sum_{l=1}^{k-1} \textbf{u}_l \textbf{u}_l^T \right) \textbf{T}_\alpha(\textbf{u}_{k})$\;
$\textbf{u}_{k, new} \longleftarrow \|\textbf{u}_{k, new}\|^{-1} \textbf{u}_{k, new}$\;
$\Delta \longleftarrow  \min\{ \| \textbf{u}_{k, new}  -  \textbf{u}_{k} \| ,  \| \textbf{u}_{k, new}  + \textbf{u}_{k}\|  \}$\;
$\textbf{u}_{k} \longleftarrow \textbf{u}_{k, new}$\;
}
\caption{Deflation-based signal separation \label{algo:sep}}

\end{algorithm}

\begin{remark}\label{rem:init1}
To ensure that the signals are found in the right order and to guarantee hence affine equivariance the initial value must be chosen in practise carefully. Following \citet{nordhausen2011deflation}, one possible initial value for the vector $\textbf{u}_k$ is obtained by first searching the FOBI \citep{cardoso1989source} or kJADE \citep{MiettinenNordhausenOjaTaskinen:2013} solution for $\textbf{x}$  and then taking that row of the estimated rotation which gives the $k$th largest value for the objective function $G_\alpha$.
\end{remark}
%\begin{itemize}
%\item[0.] Initialize $\textbf{u}_k$.
%\item[1.] Set $\textbf{u}_k \leftarrow \left(\textbf{I}_p - \sum_{j=1}^{k-1} \textbf{u}_j \textbf{u}_j^T \right) \textbf{T}_k$.
%\item[2.] Set $\textbf{u}_k \leftarrow \|\textbf{u}_k\|^{-1} \textbf{u}_k$.
%\item[3.] Repeat from step 1 until convergence.
%\end{itemize}

Besides the algorithm, the estimating equations provide a way to derive the asymptotic behavior of the estimated signal separation functional $\hat{\textbf{W}}$ in the case $\boldsymbol{\Omega} = \textbf{I}_p$. As discussed in Section \ref{subsec:ssf}, considering this special case only is sufficient as the estimate $\hat{\textbf{W}}$ is affine equivariant. %The version of the following theorem using the third described set of assumptions is novel, the first two sets leading into special cases of Corollary 1 in \cite{Ollila2010}.

\begin{theorem} \label{theo:sep_asymp}

Let $\textbf{x}_1,\ldots,\textbf{x}_n$ be a random sample from the independent component model in \eqref{eq:icm_icmodel} with $\boldsymbol{\Omega} = \textbf{I}_p$.
Assume that the eighth moments exist and that $\min_{1\le j \le d} \{\alpha\gamma_j^2+(1-\alpha)\kappa_j^2 \}> 0$. Then there exists a sequence of solutions $\hat{\textbf{W}}$ such that $\hat{\textbf{W}}\rightarrow_P (\textbf{I}_d, \textbf{0})$ and the limiting distribution of $\sqrt{n} \, vec(\hat{\textbf{W}} - (\textbf{I}_d, \textbf{0}))$ is multivariate normal with mean vector $\textbf{0}$ and the following asymptotic variances.
\begin{align*}
ASV(\hat{w}_{kl}) &= A_l + 1, \quad&l < k, \\
ASV(\hat{w}_{kk}) &= D_k, \quad & \\
ASV(\hat{w}_{kl}) &= A_k, \quad &l > k.
\end{align*}
\end{theorem}

\begin{figure}[t]
\begin{center}
\begin{tikzpicture}
\matrix (m) [left delimiter=(, right delimiter=), matrix of math nodes]{
 D_1 & A_1 & A_1 & \cdots & A_1 & A_1 & \cdots & A_1 & A_1 \\
A_1 + 1 & D_2 & A_2 & \cdots & A_2 & A_2 & \cdots & A_2 & A_2 \\
A_1 + 1 & A_2 + 1 & D_3 & \cdots & A_3 & A_3 & \cdots & A_3 & A_3 \\
\vdots & \vdots & \vdots & \ddots & \vdots & \vdots & \cdots & \vdots & \vdots \\
A_1 + 1 & A_2 + 1 & A_3 + 1 & \cdots & D_d & A_d & \cdots & A_d & A_d \\
};
\draw (m-5-1.south west)--($(m-1-1.north west)+(-2.0ex,0)$)--(m-1-9.north east)--(m-5-9.south east)--(m-5-1.south west);
\draw (m-5-2.south west)--(m-2-1.north east)--(m-2-9.north east);
\draw (m-5-3.south west)--(m-3-2.north east)--(m-3-9.north east);
\draw (m-5-5.south west)--(m-5-5.north west)--(m-5-9.north east);
\draw [dashed] (m-5-5.south east)--(m-1-5.north east);
\draw [decorate,decoration={brace,amplitude=5pt}]
($(m-1-1.north west)+(-2.0ex,1.0ex)$)--($(m-1-5.north east)+(-0.2ex,1.0ex)$) node [black,midway,yshift=16pt,xshift=2pt]
{\footnotesize $\hat{\textbf{W}}_1$};
\draw [decorate,decoration={brace,amplitude=5pt}]
($(m-1-6.north west)+(0.2ex,1.0ex)$)--($(m-1-9.north east)+(0,1.0ex)$) node [black,midway,yshift=16pt,xshift=2pt]
{\footnotesize $\hat{\textbf{W}}_2$};
\end{tikzpicture}
\end{center}
\caption{The asymptotic variances of the individual elements of the estimated deflation-based projection pursuit functional $\hat{\textbf{W}}$.}
\label{fig:asymp_sep}
\end{figure}

For $d=p$, the assumptions concerning third and fourth cumulants may be slightly relaxed as one of the values $\alpha\gamma_j^2+(1-\alpha)\kappa_j^2$ can then be zero. Also, if $\alpha=1$ it is sufficient that only the sixth moments exist.

The first expression of Theorem \ref{theo:sep_asymp} describes the asymptotic variances of the lower diagonal part of $\hat{\textbf{W}}$ and the third expression the asymptotic variances of the upper diagonal part, see Figure \ref{fig:asymp_sep}. Three interesting remarks include: firstly, the asymptotic variances of the first $k$ rows of the unmixing matrix estimate $\hat{\textbf{W}}$ do not depend  on the distribution of any of the further signals; secondly, the asymptotic variances of the rows of the submatrix $\hat{\textbf{W}}_2$ depend only on the distribution of that particular signal and, thirdly, if the $k$th independent component has zero skewness (excess kurtosis) then for all $\alpha \neq 1$ ($\alpha \neq 0)$ the value of $A_k$ is the same, that is, if there is no skewness (excess kurtosis) no cost has to be paid in the asymptotic variances for a ``wrong'' choice of $\alpha \in (0, 1)$. 
%For discussion and comparison of the methods based on their asmyptotic variances see Sections \ref{sec:simu} and \ref{sec:sub}.
%Interestingly, for both the upper and lower diagonal parts the expressions depend only on the properties of the signal distributions. Of all the approaches discussed in the paper this behavior is unique to the deflation-based method.

%\begin{tikzpicture}
%\draw[step=0.5cm, lightgray, very thin] (0,0) grid (6,3);
%\draw (0,0) -- (0,3) -- (6,3) -- (6,0) -- (0,0);
%\draw (0.5,0.0) -- (0.5,2.5) -- (6.0,2.5);
%\draw (1.0,0.0) -- (1.0,2.0) -- (6.0,2.0);
%\draw (1.5,0.0) -- (1.5,1.5) -- (6.0,1.5);
%\draw (2.0,0.0) -- (2.0,1.0) -- (6.0,1.0);
%\end{tikzpicture}

\section{ESTIMATING THE SIGNALS SIMULTANEOUSLY}\label{sec:sym}

In this section we extend on the previous one by exploring the simultaneous estimation of all signal components. The methodology then corresponds to the use of $L_2$-norm
and convex combinations of criterion functions in symmetric FastICA.
%and is apart from the two remarks mentioned in the introduction previously unexplored.
As in Section \ref{sec:sep}, we first provide the justification for the validity of the approach in the form of the following inequality.

\begin{theorem}\label{theo:sym_ineq}
Assume the model in \eqref{eq:icm_icmodel} and, for a chosen $\alpha\in[0,1]$, let $\textbf{U} = (\textbf{u}_1,\ldots,\textbf{u}_{d})^T$ be the matrix of Theorem \ref{theo:icm_rotate}. Then
\[
\sum_{k=1}^d G_\alpha(\textbf{v}_k) \leq \sum_{k=1}^d G_\alpha(\textbf{u}_k),
\]
for all matrices $\textbf{V} = (\textbf{v}_1,\ldots, \textbf{v}_d)^T \in \mathcal{U}^{d \times p}$ with orthonormal rows.
\end{theorem}
%\sum_{k=1}^d \left( \alpha \gamma^2(\textbf{u}_k^T \textbf{z}) + (1-\alpha) \kappa^2(\textbf{u}_k^T \textbf{z}) \right)
%\leq \sum_{k=1}^d \left( \alpha \gamma_k^2 + (1-\alpha) \kappa_k^2 \right),
The inequality in Theorem \ref{theo:sym_ineq} says that the true value of $\textbf{U}$ is a global maximizer of the sum of the objective functions $G_\alpha$ for its individual rows suggesting again an optimization procedure for estimating it. In the context of independent component analysis \citet{comon1994independent} calls projection indices that satisfy inequalities such as the one in Theorem \ref{theo:sym_ineq} \textit{contrasts}, see also \citet{moreau2001generalization}. Both of them show that in general any cumulants of order three or higher can be used in independent component analysis as contrasts.% The previous suggests the following strategy for searching for the signal subspace.

\begin{definition}\label{def:sym_func}
 For a chosen $\alpha \in [0, 1]$, the symmetric projection pursuit functional based on convex combination of squared third and fourth cumulants is a functional $\textbf{W}(F_\textbf{x}) = \textbf{U} \boldsymbol{\Sigma}^{-1/2}$ where $\boldsymbol{\Sigma} = Cov(\textbf{x})$ and the matrix $\textbf{U} = (\textbf{u}_1,\ldots,\textbf{u}_d)^T$ having orthonormal rows satisfies
\[ \textbf{U} = \underset{\textbf{U} \textbf{U}^T = \textbf{I}_d}{\emph{argmax}} \left(  \sum_{k=1}^d G_\alpha(\textbf{u}_k) \right). \]
\end{definition}

Recall that in symmetric FastICA utilizing third or fourth cumulants one finds $\textbf{U} \in \mathcal{U}$ that maximizes either  $ \sum_{k=1}^p |\gamma(\textbf{u}_k^T \textbf{x}_{st})|$ or $ \sum_{k=1}^p |\kappa(\textbf{u}_k^T \textbf{x}_{st})|$, the $L_1$-norm of the cumulant vector. Since $L_p$-norms are not equal in $\mathbb{R}^{d}$, $d > 1$, the $L_2$-based method of Definition \ref{def:sym_func} differs from symmetric FastICA also for the marginal weights $\alpha \in \{0, 1\}$. Compare this to the discussion after Definition \ref{def:sep_func} in the previous section.

Like in the previous section it is again straightforwardly seen that the proposed functional is affine equivariant implying the following.

% and  Theorem \ref{theo:sym_ineq} implies the following.

\begin{lemma}\label{lem:symm_ae}
The symmetric projection pursuit functional $\textbf{W} (F_\textbf{x})$ in Definition \ref{def:sym_func} is a signal separation functional for every $\alpha \in [0, 1]$.
\end{lemma}

%Depending on the number of signals the simultaneous extraction of signals requires two different algorithms. If $d < p$ one uses Algorithm \ref{algo:sym1} and if $d = p$ one uses \ref{algo:sym2}.

The derivation of the following fixed-point algorithm is again postponed to the Appendix. Defining $\textbf{T}_\alpha(\textbf{U}) := (\textbf{T}_\alpha(\textbf{u}_1) ,\ldots, \textbf{T}_\alpha(\textbf{u}_d))^T$ we then have the following.

%\begin{algorithm}[H]
%\SetAlgoLined
%Initialize $\textbf{U}_{}$\;
%$\Delta \longleftarrow \infty$\;
%\While{$ \Delta > \epsilon$}{
%\For{$k\in \{1,\ldots,d\}$}{
%$\textbf{u}_{k, new} \longleftarrow \left(\textbf{I}_p - \sum_{l=1}^{k-1} \textbf{u}_{l, new} \textbf{u}_{l, new}^T - \sum_{l=k+1}^{d} \textbf{u}_{l} \textbf{u}_{l}^T \right) \textbf{T}_\alpha(\textbf{u}_{k})$\;
%$\textbf{u}_{k, new} \longleftarrow \|\textbf{u}_{k, new}\|^{-1} \textbf{u}_{k, new}$\;
%}
%$\textbf{U}_{new} \longleftarrow (\textbf{u}_1,...,\textbf{u}_d)^T$\;
%$\Delta \longleftarrow  \min_{ \textbf{J}\in \mathcal{J} } \| \textbf{U}_{new}  -  \textbf{J} \textbf{U}_{}  \| $\;
%$\textbf{U}_{} \longleftarrow \textbf{U}_{new}$\;
%}
%\caption{Symmetric signal separation, $d < p$ \label{algo:sym1}}
%\end{algorithm}
%\begin{algorithm}[H]
%\SetAlgoLined
%Initialize $\textbf{U}_{}$\;
%$\Delta \longleftarrow \infty$\;
%\While{$ \Delta > \epsilon$}{
%$\textbf{U}_{new} \longleftarrow \textbf{T}_\alpha(\textbf{U}_{}) [ \textbf{T}_\alpha(\textbf{U}_{})^T \textbf{T}_\alpha(\textbf{U}_{}) ]^{-1/2}$\;
%$\Delta \longleftarrow \min_{\textbf{J}\in \mathcal{J} } \|  \textbf{U}_{new}  -  \textbf{J} \textbf{U}_{}  \| $\;
%$\textbf{U}_{} \longleftarrow \textbf{U}_{new}$\;
%}

\begin{algorithm}[H]
\SetAlgoLined
Initialize $\textbf{U}_{}$\;
$\Delta \longleftarrow \infty$\;
\While{$\Delta > \epsilon$}{
$\textbf{U}_{new} \longleftarrow [ \textbf{T}_\alpha(\textbf{U}_{}) \textbf{T}_\alpha(\textbf{U}_{})^T ]^{-1/2} \textbf{T}_\alpha(\textbf{U}_{})$\;
$\Delta \longleftarrow \min_{\textbf{J}\in \mathcal{J} } \|  \textbf{U}_{new}  -  \textbf{J} \textbf{U}_{}  \| $\;
$\textbf{U}_{} \longleftarrow \textbf{U}_{new}$\;
}
\caption{Symmetric signal separation \label{algo:sym1}}
\end{algorithm}
\begin{remark}\label{rem:init2}
Although the choice of the initial value seems less crucial in this case we advice again to use as initial value for the matrix $\textbf{U}$ the appropriately ordered directions based on FOBI or kJADE.
\end{remark}
%\[\textbf{U} \leftarrow  \textbf{T} (\textbf{T}^T \textbf{T})^{-1/2}, \]
%where $\textbf{T}= (\textbf{T}_1,..., \textbf{T}_p)^T$ and $\textbf{T}_k$ is as in Section \ref{subsec:sep}.
Unlike in Algorithm \ref{algo:sep} the order of the extracted signals is not fixed in Algorithm \ref{algo:sym1}. Thus, to choose the most important signals (in the sense of the objective function) one has to compare the values $G_\alpha(\textbf{u}_k)$, $k=1,\ldots,d$, post-extraction.

\if0\blind{Besides the above algorithm, the estimating equations in the Appendix further provide us with the following novel asymptotic variances for the elements of the estimate $\hat{\textbf{W}}$ in the case $\boldsymbol{\Omega} = \textbf{I}_p$.}\fi

\if1\blind{Besides the above algorithm, the estimating equations in the Appendix further provide us with the following novel asymptotic variances for the elements of the estimate $\hat{\textbf{W}}$ in the case $\boldsymbol{\Omega} = \textbf{I}_p$.}\fi
%Forming $\textbf{U}_1$ via this manner guarantees that the estimate of $\hat{\textbf{W}}_1$ obtained from Algorithm \ref{algo:sym} is affine equivariant in the sense of Definition \ref{def:ssf}.

\begin{theorem}\label{theo:sym_asymp}

Let $\textbf{x}_1,\ldots,\textbf{x}_n$ be a random sample  from the independent component model in \eqref{eq:icm_icmodel} with $\boldsymbol{\Omega}=\textbf{I}_p$.
Assume that the eighth moments exist and that $\min_{1\le j \le d} \{\alpha\gamma_j^2+(1-\alpha)\kappa_j^2 \}> 0$.
  Then there exists a sequence of solutions $\hat{\textbf{W}}$ such that $\hat{\textbf{W}} \rightarrow_P (\textbf{I}_d, \textbf{0})$ and the limiting distribution of $\sqrt{n} \, vec(\hat{\textbf{W}} - (\textbf{I}_d, \textbf{0}))$ is multivariate normal with mean vector $\textbf{0}$ and the following asymptotic variances.
\begin{align*}
ASV(\hat{w}_{kl}) &= B_{kl}, \quad &l \leq d, l \neq k, \\
ASV(\hat{w}_{kk}) &= D_{k}, \\
ASV(\hat{w}_{kl}) &= A_{k}, \quad &l > d.
\end{align*}
\end{theorem}

\begin{figure}[t]
\begin{center}
\begin{tikzpicture}
\matrix (m) [left delimiter=(, right delimiter=), matrix of math nodes]{
D_1 & B_{12} & B_{13} & \cdots & B_{1d} & A_1 & \cdots & A_1 & A_1 \\
B_{21} & D_2 & B_{23} & \cdots & B_{2d} & A_2 & \cdots & A_2 & A_2 \\
B_{31} & B_{32} & D_3 & \cdots & B_{3d} & A_3 & \cdots & A_3 & A_3 \\
\vdots & \vdots & \vdots & \ddots & \vdots & \vdots & \cdots & \vdots & \vdots \\
B_{d1} & B_{d2} & B_{d3} & \cdots & D_d & A_d & \cdots & A_d & A_d \\
};
\draw (m-5-1.south west)--($(m-1-1.north west)+(-0.4ex,0)$)--(m-1-9.north east)--(m-5-9.south east)--(m-5-1.south west);

\draw (m-5-2.south west)--(m-5-2.north west);
\draw (m-3-2.south west)--(m-1-2.north west);

\draw (m-2-1.north west)--(m-2-3.north east);
\draw (m-2-5.north west)--(m-2-9.north east);

\draw (m-5-3.south west)--(m-5-3.north west);
\draw (m-3-2.south east)--(m-1-3.north west);

\draw (m-3-1.north west)--(m-2-3.south east);
\draw (m-3-5.north west)--(m-3-9.north east);

%\draw (m-5-4.south west)--(m-5-4.north west);
%\draw (m-3-4.south west)--(m-1-4.north west);

%\draw (m-4-1.north west)--(m-4-3.north east);
%\draw (m-4-5.north west)--(m-4-9.north east);

\draw (m-5-5.south west)--(m-5-5.north west);
%\draw (m-3-5.south west)--(m-1-5.north west);

%\draw (m-5-1.north west)--(m-5-3.north east);
\draw (m-5-5.north west)--(m-5-9.north east);

\draw [dashed] (m-5-5.south east)--(m-1-5.north east);
\draw [decorate,decoration={brace,amplitude=5pt}]
($(m-1-1.north west)+(-2.0ex,1.0ex)$)--($(m-1-5.north east)+(-0.2ex,1.0ex)$) node [black,midway,yshift=16pt,xshift=2pt]
{\footnotesize $\hat{\textbf{W}}_1$};
\draw [decorate,decoration={brace,amplitude=5pt}]
($(m-1-6.north west)+(0.2ex,1.0ex)$)--($(m-1-9.north east)+(0,1.0ex)$) node [black,midway,yshift=16pt,xshift=2pt]
{\footnotesize $\hat{\textbf{W}}_2$};
\end{tikzpicture}
\end{center}
\caption{The asymptotic variances of the individual elements of the estimated symmetric projection pursuit functional $\hat{\textbf{W}}$.}
\label{fig:asymp_sym}
\end{figure}

Again, if $d=p$ it is sufficient that at most one of the values $\alpha\gamma_j^2+(1-\alpha)\kappa_j^2$ is zero and for $\alpha=1$ it is sufficient that the sixth moments exist.

A visual description of Theorem \ref{theo:sym_asymp} is given in Figure \ref{fig:asymp_sym}. Comparing it to Figure \ref{fig:asymp_sep} shows one fundamental difference between the two approaches; every off-diagonal element of $\hat{\textbf{W}}_1$ has asymptotic variance depending on both the row and column index signals. However, the matrix $\hat{\textbf{W}}_2$ has equal asymptotic behavior for both methods. This aspect is further discussed in Section \ref{sec:simu}. Analogously to the deflation-based method, if the $k$th and $l$th independent components both have zero skewness (excess kurtosis) then all choices of $\alpha \neq 1$ ($\alpha \neq 0)$ yield the same value for $B_{kl}$. %We have not given a separate visualization for the case $d = p$ as it is obtained simply by letting $\hat{\textbf{W}}_1$ grow in Figure \ref{fig:asymp_sym} to the size $p \times p$, squeezing $\hat{\textbf{W}}_2$  into nothingness.

\begin{remark}\label{cor:asymp_equiv}
Comparing the asymptotic variances in Theorem \ref{theo:sym_asymp} in the marginal case $d = p$ and $\alpha = 0$ with those of JADE in \cite{miettinen2014fourth} shows that the two are equal. Thus the symmetric $L_2$-based projection pursuit using fourth cumulants provides with a lighter computational load the same asymptotic accuracy as given by the classical JADE method.
\end{remark}

\section{LIMITING EFFICIENCY OF THE SIGNAL SEPARATION ESTIMATE} \label{sec:simu}

\subsection{The two parts of the signal separation estimate}

Consider the division of the signal separation estimate into the two parts shown in Figures \ref{fig:asymp_sep} and \ref{fig:asymp_sym}, $\hat{\textbf{W}} = (\hat{\textbf{W}}_1, \hat{\textbf{W}}_2)$. Assuming $\mathbf{\Omega}=\textbf{I}_p$, which is again sufficient because of the affine equivariance, we can then write.% The comparison of the two methods is divided into two sections, this one discussing the estimation of individual signals and Section \ref{sec:sub} considering the estimation of the signal subspace as a whole. This division is justified by first noting that the case $\mathbf{\Omega}=\textbf{I}_p$ is sufficient because of the affine equivariance and then considering an estimated signal separation functional $\hat{\textbf{W}} = (\hat{\textbf{W}}_1, \hat{\textbf{W}}_2)$ and writing
\[\hat{\textbf{W}}\textbf{x} = \hat{\textbf{W}}_1 \textbf{s} + \hat{\textbf{W}}_2 \textbf{n}. \]
Thus the variation of $\hat{\textbf{W}}_1$ around $\textbf{I}_d$ tells how well the individual signals are separated from each other and the variation of $\hat{\textbf{W}}_2$ around zero matrix informs of the ability to separate between the signal and noise subspaces. The asymptotic variances of the elements of these matrices therefore measure the accuracy of the respective separations. Using the affine equivariance of the estimate we can prove the following theorem implying that $\hat{\textbf{W}}_1$ and $\hat{\textbf{W}}_2$ are asymptotically independent.% As $\hat{\textbf{W}}_2$   $\hat{\textbf{W}}_2\textbf{V}$ has the same distribution for all $(p-d)\times(p-d)$-dimensional orthogonal matrices $\textbf{V}$,  one can easily show  the following theorem stating that $\hat{\textbf{W}}_1$ and $\hat{\textbf{W}}_2$ are asymptotically independent.
\begin{theorem}\label{theo:as independence}
The covariance matrix of the limiting distribution of $\sqrt{n} \, vec(\hat{\textbf{W}} - (\textbf{I}_d, \textbf{0}))=vec((\sqrt{n}(\hat{\textbf{W}_1} -\textbf{I}_d),\sqrt{n}\hat{\textbf{W}}_2 ))$
is block-diagonal with $p-d+1$ blocks of sizes $d^2\times d^2$, $d\times d,\ldots,d\times d$. The last $p-d$ block covariance matrices are the same.
\end{theorem}

\subsection{Comparison of individual signal estimates}

\begin{figure}[t]
    \begin{center}
    \includegraphics[width=0.8\textwidth]{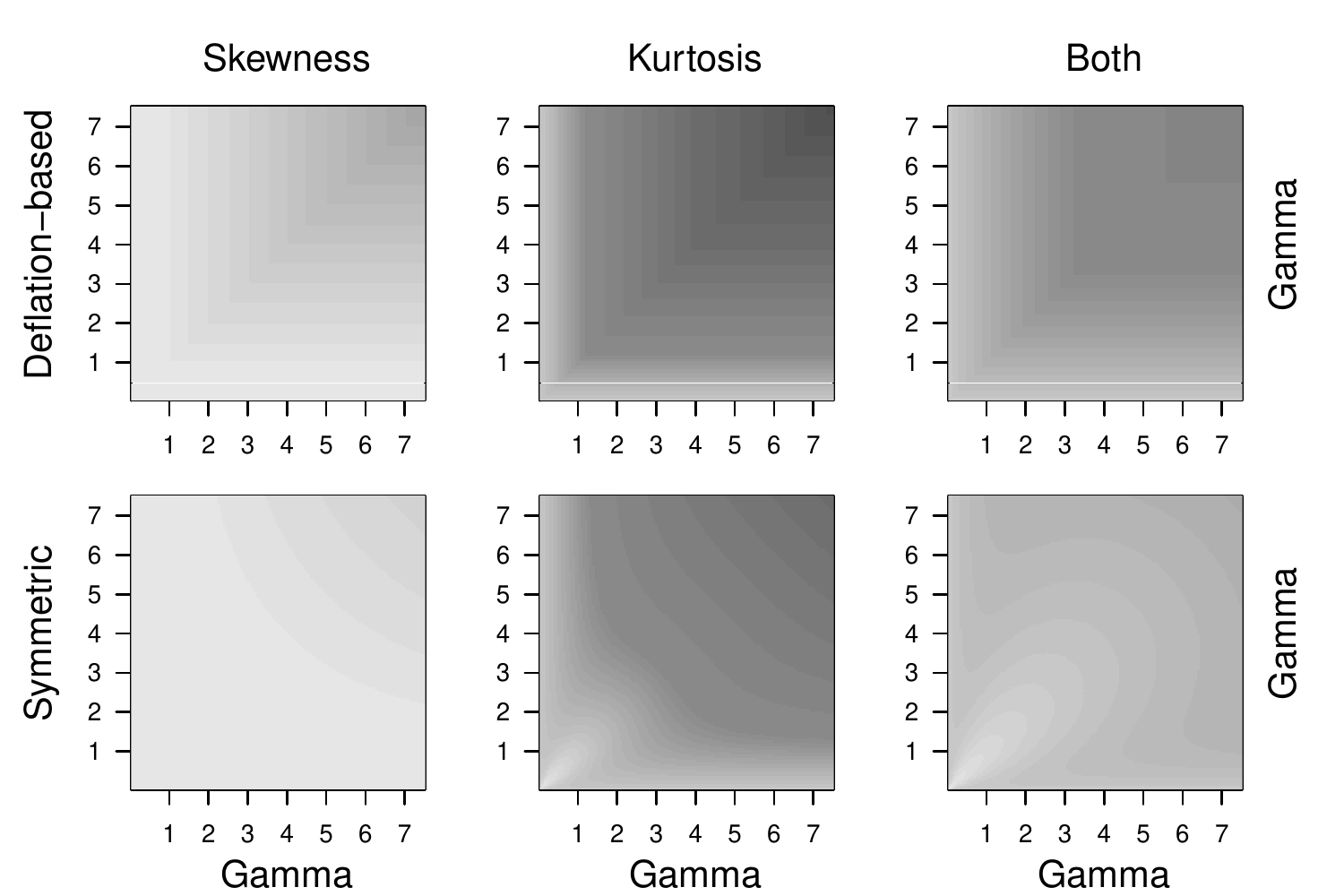}
    \caption{Contour plots of $ASV(\hat{w}_{12}) + ASV(\hat{w}_{21})$ for different combinations of methods and weighting $\alpha$ when both the $x$-axis and the $y$-axis independent components have a gamma distribution. The darker the color, the larger the sum of variances.}
    \label{fig:sim_11}
    \end{center}
\end{figure}

\begin{figure}[t]
    \centering
    \includegraphics[width=0.8\textwidth]{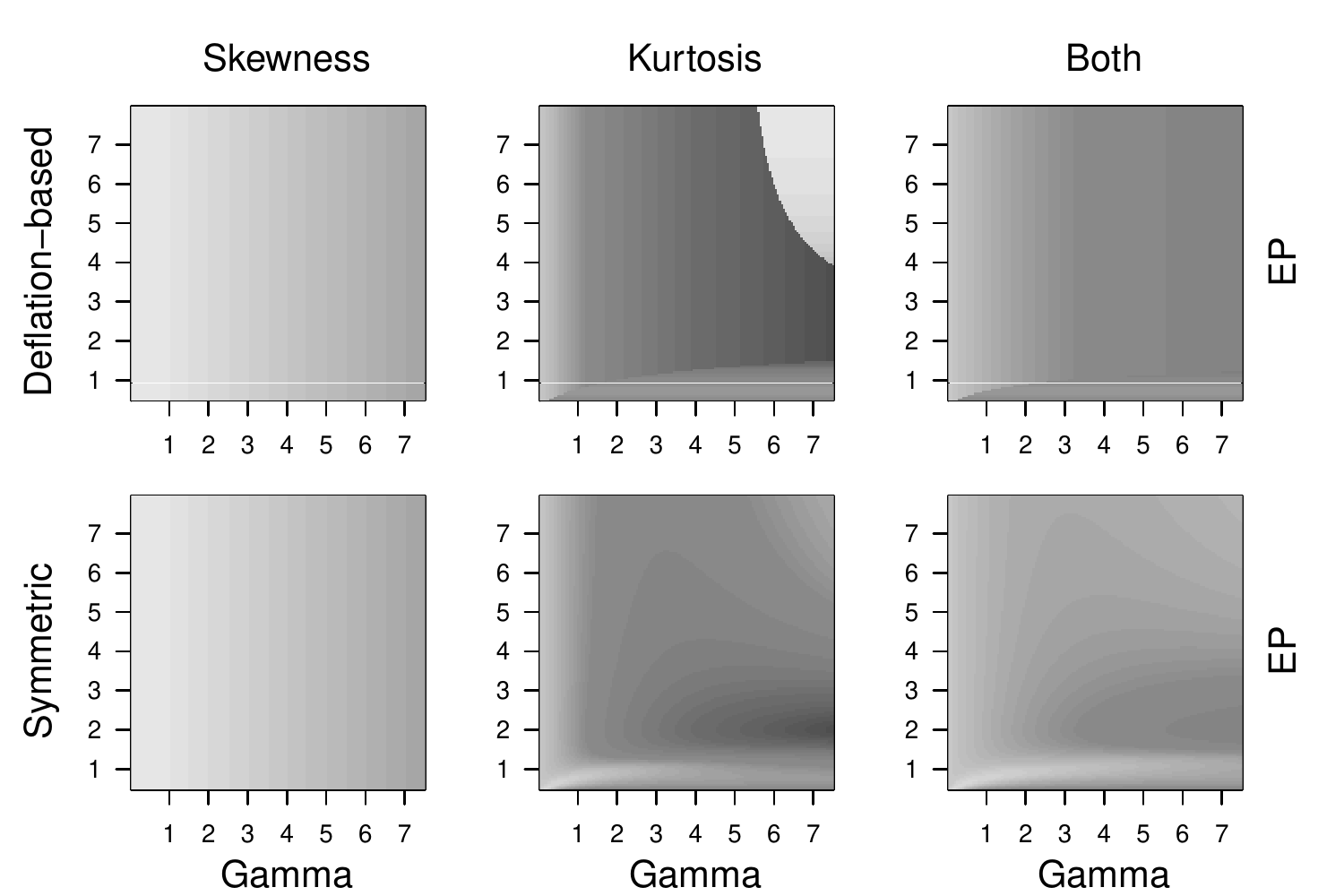}
    \caption{Contour plots of $ASV(\hat{w}_{12}) + ASV(\hat{w}_{21})$ for different combinations of methods and weighting $\alpha$ when the $x$-axis independent component has a gamma distribution and the $y$-axis independent component has an exponential power distribution. The darker the color, the larger the sum of variances.}
    \label{fig:sim_12}
\end{figure}

Both discussed projection pursuit methods produce consistent estimates, $\sqrt{n} \ vec(\hat{\textbf{W}}_1 - \textbf{I}_d)\rightarrow_d N_{d^2}(\textbf{0}, \boldsymbol{\Phi}_1)$, and their comparison should thus be made using the asymptotic covariance matrix $\boldsymbol{\Phi}_1$. A global measure of variation is given by
\begin{eqnarray*}
tr(\boldsymbol{\Phi}_1)&=& \sum_{k=1}^d \sum_{l=1}^d ASV(\hat{w}_{kl}) \\
&=& \sum_{k=1}^d ASV(\hat{w}_{kk}) +\sum_{k=1}^{d-1}\sum_{l=k+1}^d \left(ASV(\hat{w}_{kl}) + ASV(\hat{w}_{lk})\right),
\end{eqnarray*}
where $\sum_{k=1}^d ASV(\hat{w}_{kk})$ is the same for both estimates and for all $\alpha$, see Theorems \ref{theo:sep_asymp} and \ref{theo:sym_asymp}. Conveniently, for both methods $ASV(\hat{w}_{kl})$, for $k\ne l$, depends only on the distributions of the components $z_k$ and $z_l$ and is not affected by the dimensions $d$ and $p$. Consequently it is sufficient to consider, as in \cite{miettinen2014fourth}, only the values $ V_{12} := ASV(\hat{w}_{12}) + ASV(\hat{w}_{21})$ for different choices of two marginal distributions when measuring their influence to the asymptotic performance. For the symmetric projection pursuit we have $V_{12} = B_{12} + B_{21}$ and, assuming that the component $z_j$, $j=1,2$, is extracted first, for the deflation-based projection pursuit we have $V_{12} = 2 A_j + 1$.

The two marginal distributions were chosen as standardized versions of either the exponential power distribution, $\mbox{EP}(\lambda)$, or the gamma distribution, $\Gamma (\lambda)$, with the respective densities
\[
f(z) \propto e^{-\tau_1 |z|^{\lambda}}, \ z \in \mathbb{R} \ \ \mbox{and} \ \ f(z) \propto z^{\lambda-1} e^{-\tau_2 z}, \ z>0,
\]
with $\tau_1,\tau_2>0$ and a positive shape parameter $\lambda$. The quantity $V_{12}$ depends on the marginal distributions only through their shape parameters $\lambda$, see \cite{miettinen2014fourth} for more details. For both methods, we distinguished the versions using third cumulants only, $\alpha=1$, fourth cumulants only, $\alpha=0$, and a convex combination with the weight $\alpha = 0.8$, see Section \ref{sec:sub} for motivating this choice. We then computed the value of $V_{12}$ for different combinations of distribution families and shape parameters using numerical integration and Theorems \ref{theo:sep_asymp} and \ref{theo:sym_asymp} and the results are shown in Figures \ref{fig:sim_11} and \ref{fig:sim_12}. Note that we do not report the results in cases where both components come from the symmetric exponential power family as then the asymptotic variances of the estimates with $0<\alpha<1$ are the same as the asymptotic variances of the estimates with $\alpha=0$ (since skewness carries no information, both distributions being symmetric) and the results for $\alpha=0$ are already given in \cite{miettinen2014fourth}. In the figures a darker shade indicates a larger value so that the performance of a particular method is at its best in the areas of lighter color.

From the contour plots it is evident that performance-wise the methods are very close to each other, although, not counting the anomaly in the middle plot of the top row of Figure \ref{fig:sim_12}, the symmetric projection pursuit gives in general asymptotically slightly more accurate results. Also, as the distributions in Figure \ref{fig:sim_12} are respectively skewed and symmetric, skewness contains in these settings more information on the separation than kurtosis, explaining  why the left-most plots have the overall lightest shade.

\subsection{Comparison of subspace estimates}

Apart from estimating the individual signals, another interesting objective is the estimation of the whole signal subspace. Assuming $\boldsymbol{\Omega}=\textbf{I}_p$,  we approach the problem via the orthogonal projection matrix \[\textbf{P} := \textbf{W}^T (\textbf{W} \textbf{W}^T)^{-1} \textbf{W},\] giving an orthogonal projection onto the signal subspace. Our main result is then as follows.

\begin{theorem}\label{theo:proj_asymp}
Let $\hat{\textbf{W}} = (\hat{\textbf{W}}_{1},  \hat{\textbf{W}}_{2})$ be a signal separation functional estimate satisfying $\sqrt{n} (\hat{\textbf{W}} - (\textbf{I}_d, \textbf{0}))=O_P(1)$. Then the projection matrix $\hat{\textbf{P}} = \hat{\textbf{W}}{}^T (\hat{\textbf{W}} \hat{\textbf{W}}{}^T)^{-1} \hat{\textbf{W}}$  satisfies
\[\sqrt{n}\left(\hat{\textbf{P}} -
\begin{pmatrix}
\textbf{I}_d & \textbf{0} \\
\textbf{0} & \textbf{0}
\end{pmatrix}
\right) = \sqrt{n}
\begin{pmatrix}
\textbf{0} & \hat{\textbf{W}}_{2} \\
\hat{\textbf{W}}{}_{2}^T & \textbf{0}
\end{pmatrix} + o_P(1).
\]
\end{theorem}

Theorem \ref{theo:proj_asymp} essentially states that the asymptotic behavior of the estimated projection matrix $\hat{\textbf{P}}$ depends only on that
of $\hat{\textbf{W}}_{2}$, and not on how well the signals are separated from each other as measured by $\hat{\textbf{W}}_{1}$.
Since $\sqrt{n} \ vec(\hat{\textbf{W}}_2) \rightarrow_d N_{(p-d)^2}(\textbf{0}, \boldsymbol{\Phi}_2)$, a natural measure for a particular method's ability to estimate the signal subspace is then
\[ tr(\boldsymbol{\Phi}_2) = \sum_{k=1}^d \sum_{l=d+1}^p ASV(\hat{w}_{kl})= (p-d) \sum_{k=1}^d  ASV(\hat{w}_{k,d+1}),\] the sum of asymptotic variances of the elements of $\hat{\textbf{W}}_{2}$. Comparison of Figures \ref{fig:asymp_sep} and \ref{fig:asymp_sym} now easily yields the conclusion that both discussed projection pursuit methods are asymptotically equally adept at estimating the signal subspace. That is, if we are only interested in separating the signal from the noise, it does not matter asymptotically whether we use the method of Definition \ref{def:sep_func} or \ref{def:sym_func}.

\subsection{Finite-sample performance}

We next compare the methods' performances as the relative amount of noise in the model is increased. Assuming $d < p$, one possible measure for the accuracy of the separation is given by%In the case of identity mixing both methods satisfy $\hat{\textbf{W}} \rightarrow_P (\textbf{I}_d, \textbf{0})$ and assuming $d < p$, a suitable measure for the accuracy of the separation is then given by
\begin{align}\label{eq:nonsquare_md}
D(\hat{\textbf{W}}) := \frac{1}{\sqrt{d}} \underset{\textbf{C} \in \mathcal{C}}{\mbox{inf}} \left \| \textbf{C} \hat{\textbf{W}} \boldsymbol{\Omega} - (\textbf{I}_d, \textbf{0}) \right \|,
\end{align}
where $\boldsymbol{\Omega}$ is the true mixing matrix. The measure \eqref{eq:nonsquare_md} can be seen as an analogue of the minimum distance index (MDI) \citep{ilmonen2010new} for non-square matrices and similar techniques as used in proving Theorem 4.1 in \cite{ilmonen2012asymptotics} show that $0 \leq D(\hat{\textbf{W}}) \leq 1$, the value zero indicating perfect separation. Furthermore, the techniques used in \cite{ilmonen2012asymptotics} can be used to show that the computation of $D$ can be reduced to an optimization problem over a set of finite support and that the limiting distribution of $nd D^2(\hat{\textbf{W}})$ is a weighted sum of independent chi-squared variables with one degree of freedom. The expectation of this limiting distribution is then
 the sum of the weights, that is, the sum of asymptotic variances of the off-diagonal elements of $\hat{\textbf{W}}$, given by $tr(\boldsymbol{\Phi}_1) - \sum_{k=1}^d ASV(\hat{w}_{kk}) + tr(\boldsymbol{\Phi}_2)$.

\begin{figure}[t]
    \centering
    \includegraphics[width=1\textwidth]{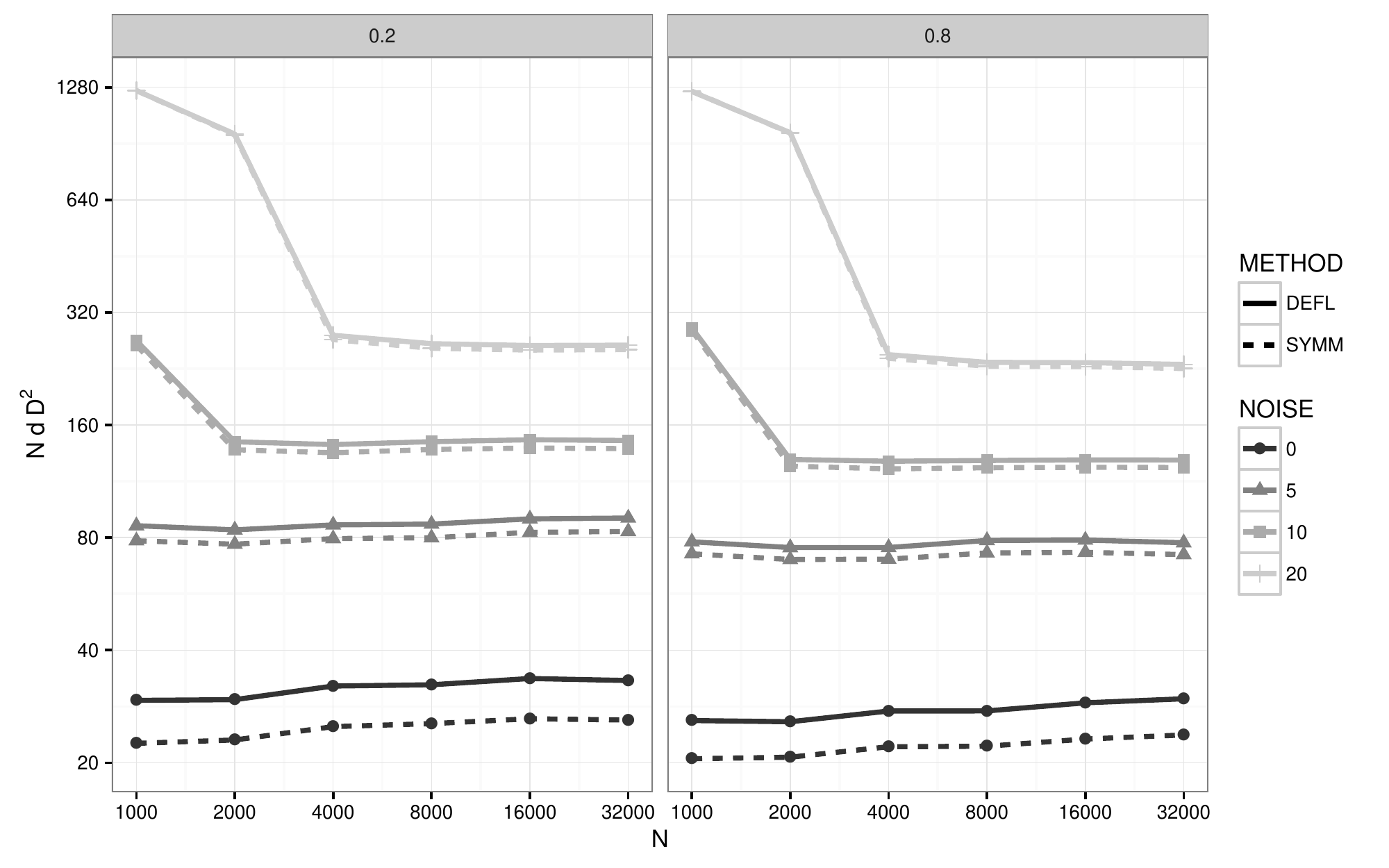}
    \caption{Average values of $ndD^2(\hat{\textbf W})$ for different noise levels when using deflation-based and symmetric approaches with two different values of $\alpha$.}
    \label{fig:sim_2}
\end{figure}

For our simulation setting we chose $d = 3$, the signal vector $\textbf{s}$ having Uniform$(0, 1)$, Exponential$(1)$ and Laplace$(0, 1)$ components each standardized to have zero mean and unit variance and the amount of Gaussian noise was taken to be $p - d = 0$, $5$, $10$, $20$. Moreover, we considered the sample sizes $n = 1000$, $2000$, $4000$, $8000$, $16000$, $32000$ and two different weightings $\alpha = 0.2$, $0.8$. The latter of these is motivated in Section \ref{sec:sub} and the former can be seen as the ``opposite'' of that. Due to the affine equivariance of the methods we used, without loss of generality, the mixing matrix $\boldsymbol{\Omega} = \textbf{I}_p$. For each of the 2000 repetitions we ran the two projection pursuit algorithms using the alternative gradients in Remark \ref{rem:grad} with the initial values based on FOBI as discussed in Remarks \ref{rem:init1} and \ref{rem:init2} and then computed the corresponding values $D(\hat{\textbf{W}})$.

The averages of $nd D^2(\hat{\textbf{W}})$ over the repetitions under the previous combinations of parameters are shown in Figure \ref{fig:sim_2}, from which three immediate remarks can be made: the choice $\alpha = 0.8$ gives uniformly better results, the symmetric projection pursuit is in all cases slightly superior to the deflation-based one and this difference gets relatively smaller and smaller by increasing the amount of noise. The latter observation is easily explained by noticing that, with respect to the expected value of the limiting distribution of $nd D^2(\hat{\textbf{W}})$, the two methods differ only by having different matrices $\boldsymbol{\Phi}_1$. As the number of noise is increased the matrix $\boldsymbol{\Phi}_2$ common to both methods gets larger while $\boldsymbol{\Phi}_1$ retains its size and its relative importance thus diminishes by the addition of noise. As a conclusion, for data with high signal-to-noise ratio (SNR) the use of the symmetric version is advocated and for data with low SNR there is not much difference between the discussed methods.

\section{SOME RELATED PROBLEMS} \label{sec:sub}

\subsection{Cluster identification}

Given that all the methods allow tuning in the form of the weighting parameter $\alpha$, a natural question is whether there exists some optimal choice of weighting for any particular choice of signal distributions. We approach this question in the context of cluster identification. Both skewness and kurtosis have been used before for similar purposes, see e.g. \cite{jones1987projection} and \cite{pena2001cluster}.

For the model, assume that $\textbf{z}$ is a mixture of two multivariate normal distributions, namely
\begin{align}\label{eq:mixture}
\textbf{z} \sim \pi \cdot {N}_p(\textbf{0}, \textbf{I}_p) + (1-\pi) \cdot {N}_p(\mu \textbf{e}_1, \textbf{I}_p),
\end{align}
standardized to have zero mean and identity covariance matrix with $\pi \in (0, 1)$, $\mu \in \mathbb{R}\backslash \{0\}$ and $\textbf{e}_1 = (1, 0,\ldots,0)^T$. The true dimension is then $d = 1$ and $\textbf{s} \in \mathbb{R}$ has a univariate bimodal distribution. Thus all signal separation functionals now consist of only one row. The current setting can also be seen as the problem of estimating the Fisher linear discrimination subspace without knowing the group membership.

\begin{figure}[t]
    \centering
    \includegraphics[width=1.0\textwidth]{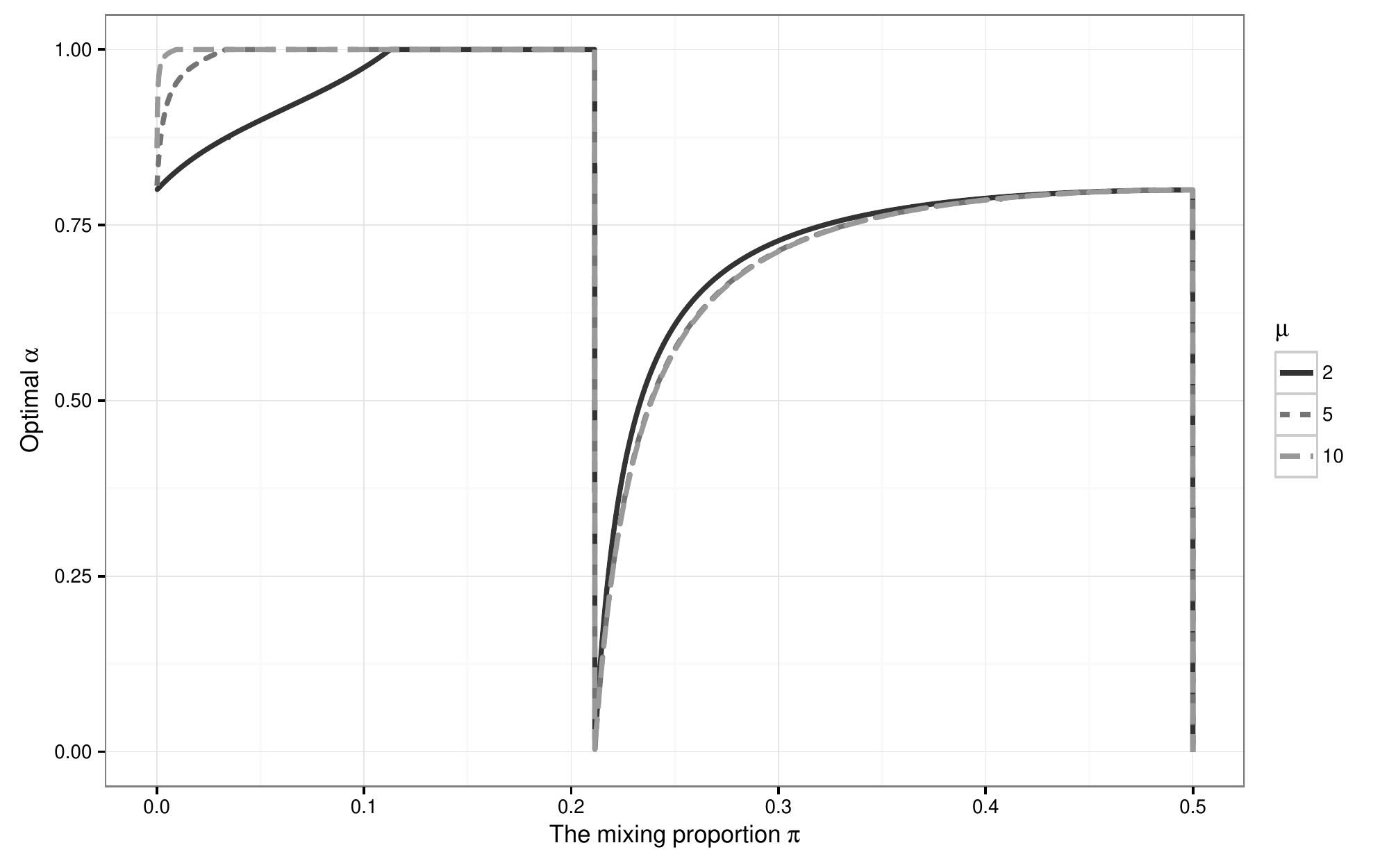}
    \caption{The optimal choices of weight $\alpha$ for different values of $\pi$ and $\mu$.}
    \label{fig:sim_discr}
\end{figure}

A motivation to use both third and fourth cumulants simultaneously in this context stems from the fact that non-trivial normal mixtures of the form \eqref{eq:mixture} can have zero skewness or zero excess kurtosis. Namely, symmetric mixtures, $\pi = 0.5$, have zero skewness and mixtures with either $\pi = \pi_0 := (3 + \sqrt{3})^{-1}$ or $\pi = 1 - \pi_0$ have zero excess kurtosis, see \cite{preston1953graphical}. This means that neither third nor fourth cumulants alone can find the latent groups for all values of $\pi$, a problem not encountered when using any non-trivial convex combination of them.

%Figures \ref{fig:asymp_sep} and \ref{fig:asymp_sym} show that the choice of method is now irrelevant from the asymptotic point of view and for a large enough value of $p$ a measure for the
For $d=1$, the deflation-based and symmetric estimates are the same and the accuracy of the estimation of the signal is provided by $A_1$, see Theorems \ref{theo:sep_asymp} and \ref{theo:sym_asymp}. We searched the values of $\alpha$ minimizing $A_1$ for $\pi \in (0, 1)$ and the three choices of $\mu = 2, 5, 10$ and the results are shown in Figure \ref{fig:sim_discr} (we only need to consider the interval $\pi \in (0, 0.5]$ due to symmetry). First, the plot shows that the choice of $\mu$ has hardly any effect on the optimal value of $\alpha$. % which is understandable since we anyway standardize the considered distribution.
Secondly, we see the two discontinuity points discussed previously, $\pi = \pi_0$ and $\pi=0.5$. And thirdly, we observe that the curve goes to zero when approaching the point $\pi_0$ from the right. This counterintuitively suggests using only fourth cumulants even though $\kappa_1 \approx 0$ in the vicinity of $\pi_0$. However, a careful examination shows that when $\pi = \pi_0 + \epsilon$ for some small $\epsilon > 0$, the quantity $A_1$ as a function of $\alpha$ indeed has a global minimum near zero but it also satisfies $\text{lim}_{\alpha \rightarrow 0+} A_1(\alpha) = \infty$. Thus for practical purposes the global minimum is too close to zero to be of any use.

While from Figure \ref{fig:sim_discr} it is clear that no single value of $\alpha$ is the best choice in every situation, the true implication of the experiment is that one should always use both cumulants instead of just one of them. However, if such universal value is needed $\alpha = 0.8$ provides a good approximation for the optimal $\alpha$ for a large set of location differences $\mu$ and mixing proportions $\pi$. This choice is further supported by its connection to the classical Jarque-Bera test statistic for normality \citep{jarque1987test}
\begin{align*}\label{eq:jarquebera}
\frac {\gamma^2}{6} + \frac{\kappa^2}{24},
\end{align*}
where the cumulants are standardized by their asymptotic standard errors in the Gaussian case, leading to weighting equivalent to choosing $\alpha = 0.8$. It is also of course highly tempting to use a classical  test statistic for normality with well-known  asymptotic behavior  in searching non-Gaussian components. The particular value $\alpha = 0.8$ also corresponds to the effective value derived in \cite{jones1987projection}.

%and as such provide a very natural measure of non-Gaussianity.

\subsection{Inference on unknown dimension $d$}

In most applications the true number of signals $d$ is unknown and must be estimated from the data. In the deflation-based approach the estimation of $d$ can be combined with the estimation of the signals as for each estimated direction the value of the objective function $G_\alpha(\textbf{u}_k)$, $k=1,\ldots,p$, is in fact a test statistic for sub-Gaussianity of the last $p-k+1$ components, a multivariate extension of the Jarque-Bera test \citep{jarque1987test}. The following algorithm outlines the basic idea of this testing procedure.

\begin{algorithm}[]
\SetAlgoLined
$\textbf{X} \longleftarrow$ a sample of size $n$\;
$\alpha \longleftarrow$ significance level\;
$N \longleftarrow$ number of normal samples drawn\;
\For{$k \in \{1,\ldots,p\}$}{
Estimate $\textbf{u}_k$ from $\textbf{X}$ with the deflation-based PP\;
\For{$j \in \{1,\ldots,N\}$}{
Draw a sample $\textbf{Y}_j$ of size $n$ from $N_{p-k+1}(\textbf{0}, \textbf{I}_{p-k+1})$\;
Estimate $\textbf{u}_1^j$ from $\textbf{Y}_j$ with the deflation-based PP\;
$G^j \longleftarrow G(\textbf{u}_1^j)$ computed from $\textbf{Y}$\;
}
$\rho \longleftarrow (1/N) \cdot \#(G(\textbf{u}_k) > G^j)$\;
\If{$\rho < 1 - \alpha$}{
\Return $k - 1$
}
}
\Return $p$
\caption{Estimation of the unknown dimension $d$ \label{algo_estd}}
\end{algorithm}
That is, since at the correct dimension $d$ the remaining components are Gaussian, we can test the null hypothesis $H_0: d = k$ by comparing the supremum of $G(\textbf{u})$ in $(\textbf{U}^T)^{\bot}$ to its simulated distribution for standard multivariate Gaussian distribution. Strict theoretical justification of this procedure is still missing but will be worked out  in a separate paper.

\begin{figure}[t]
    \centering
    \includegraphics[width=1\textwidth]{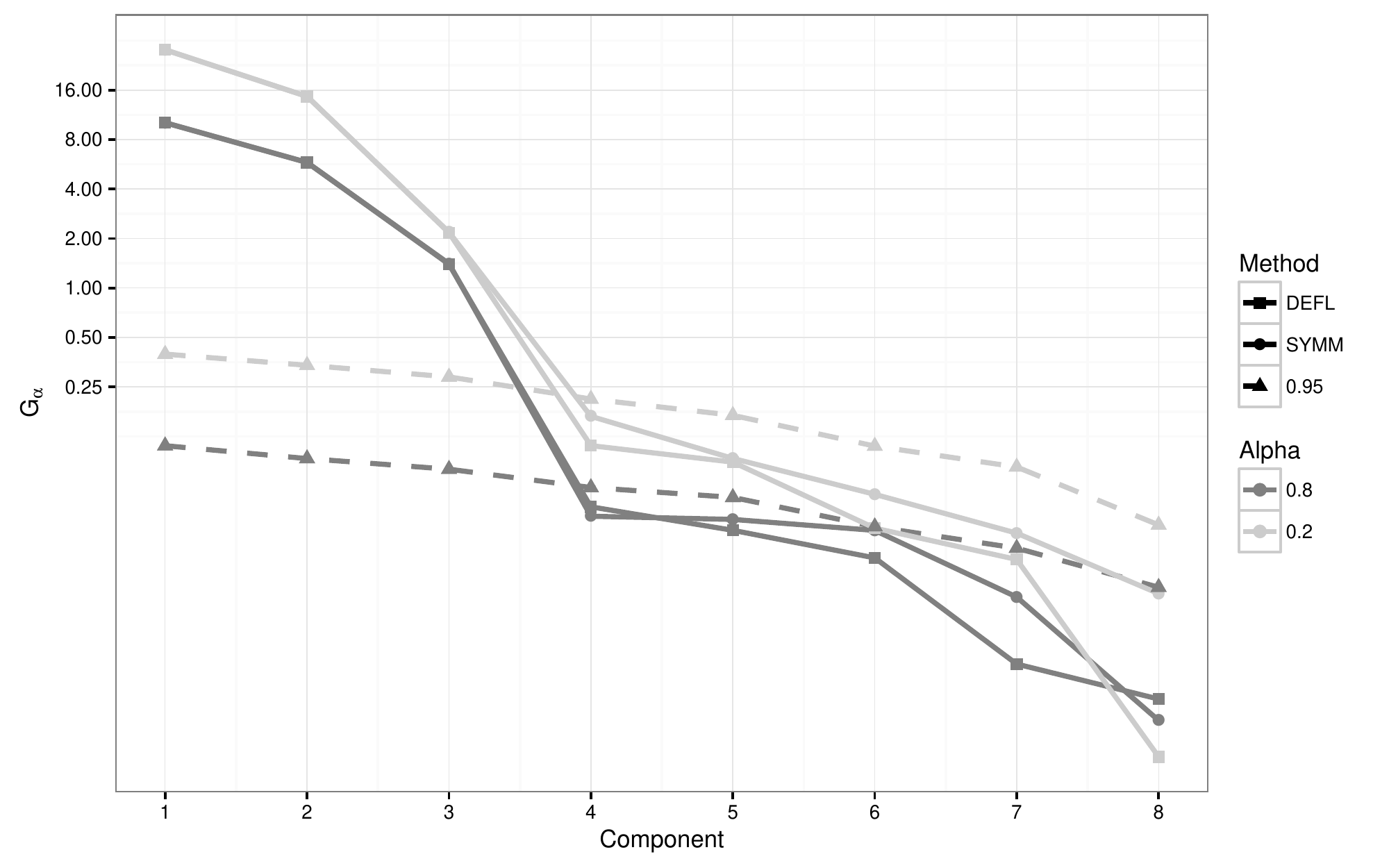}
    \caption{Screeplot for an eighth-dimensional case when using deflation-based and symmetric FastICA with two different values of $\alpha$. The dashed lines are the curves of the estimated  5 \% critical values of the tests for sub-Gaussianity.}
    \label{fig:scree}
\end{figure}

Another, na{\"i}ve way is given by using something analogous to the screeplot of PCA. That is, we first estimate full $p$ components and for each estimated row $\textbf{u}_k$, $k=1,\ldots,p$, compute the value of the objective function $G_\alpha$. Then we order the components in decreasing order with respect to these values and find a suitable cut-off point after which all remaining components have approximately zero value for the objective function. The estimate for $d$ is then the number of components with non-zero objective function value. An example of a such screeplot is given in Figure~\ref{fig:scree} where the three signals are as in the simulation setup of Section \ref{sec:simu} and the Gaussian noise dimension is five. The curves of the estimated  5 \% critical values of the tests for sub-Gaussianity based on $G_\alpha(\textbf{u}_k)$, $k=1,\ldots,p$, are provided as well (the critical values are estimated using 10000 normal samples). The figure is based on a sample of $n=2000$ and shows nicely how here the signal dimension would correctly be identified as three by all the four different approaches considered.

%Second, more rigorous approach is to extract the components one-by-one with the deflation-based projection pursuit and after each step perform a normality test on the extracted component. If the component has a distribution differing significantly from normal distribution we accept it as a signal and extract the next component. This is continued until a normal component is found and our assumptions then guarantee that all the remaining components will be normal as well. The use of some form of multiple comparison correction is advocated as well.

\subsection{A real data example}

One of the most popular applications of ICA is the demixing of vectorized  images which, while violating the assumption on i.i.d. observations, has nevertheless been shown very successful in practice, see e.g. \cite{brys2005robustification, nordhausen2008robust, HallinMehta:2015}. In fact, the algorithms then consider and analyse the marginal distribution of the intensities of an image rather than their joint  distribution. To this end, we consider the three grey-scale images available in the R-package \textit{ICS} \citep{Rics}, depicting a forest road, cat and a sheep. Each image is a matrix of size $130 \times 130$, the elements giving the intensity values of the corresponding pixels. For our setting we further simulated 12 images of same size with independent standard Gaussian noise pixels and vectorized all the images to arrive into a $130^2 \times 15$ data matrix $\textbf{X}$. Each row (``observation'') of $\textbf{X}$ contains then the intensities of a single fixed pixel across all images and each column (``variable'') the intensities of a single image. The three true images and the 15 mixed ones are shown in the Appendix.

\begin{figure}[t!]
    \centering
    \includegraphics[width=1\textwidth]{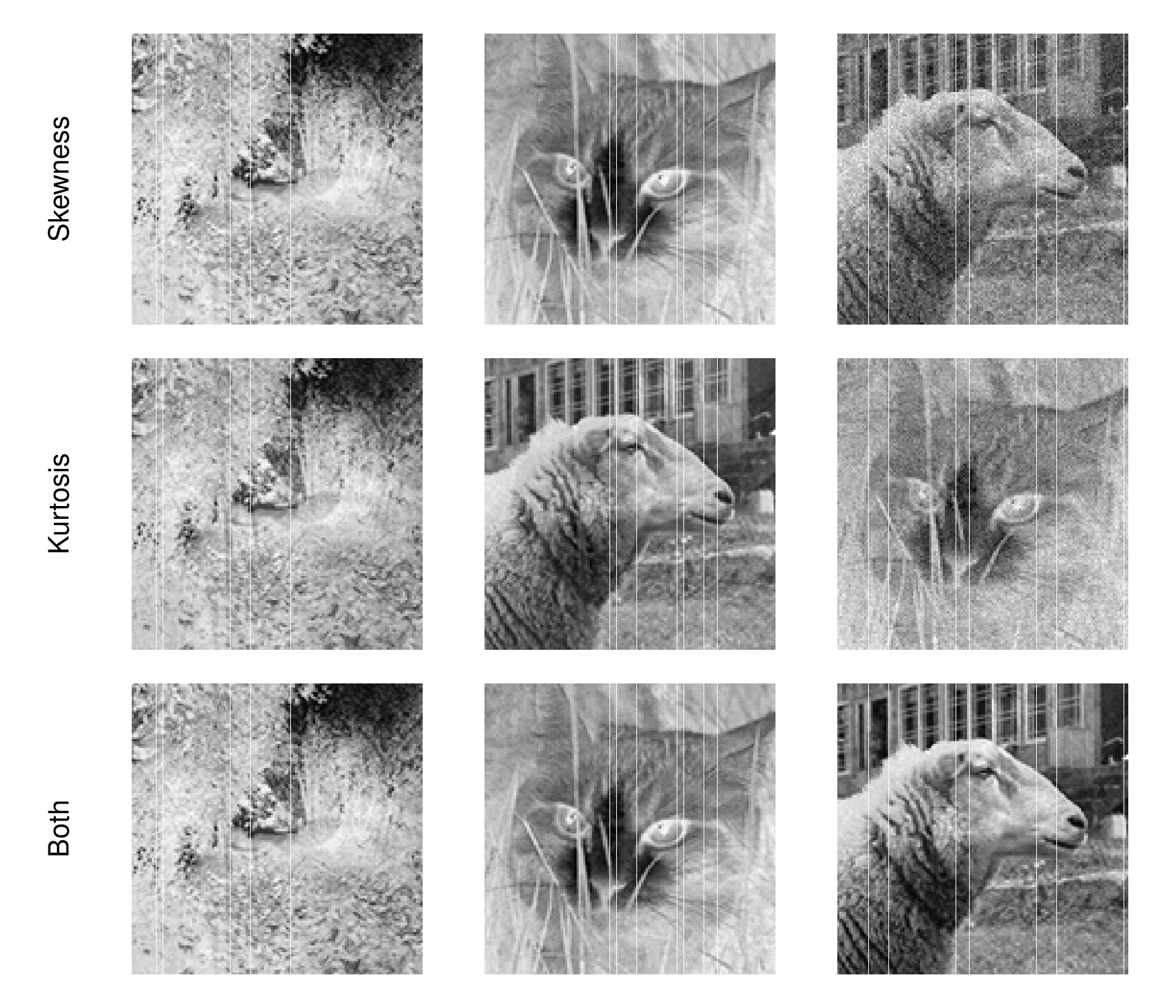}
    \caption{The three images estimated with the choices $\alpha = 0, 1, 0.8$. The leftmost image of each row is the one with the highest objective function value.}
    \label{fig:images}
\end{figure}

We next mixed the images using a random $15 \times 15$ matrix $\boldsymbol{\Omega}$ with independent standard Gaussian components as $\textbf{X} \mapsto \textbf{X} \boldsymbol{\Omega}^T$. Our objective is then to estimate from the mixture both the true images and also the number of them, $d = 3$, using deflation-based projection pursuit. Using again three different choices, $\alpha = 0, 1, 0.8$, the three estimated images with highest objective function values are shown in Figure \ref{fig:images}. The recovered images indicate that each weighting found nicely all true images but the third images found by using only a single cumulant are mixed with noise. Using a convex combination of them however found successfully all three images.

\begin{figure}[t]
    \centering
    \includegraphics[width=1\textwidth]{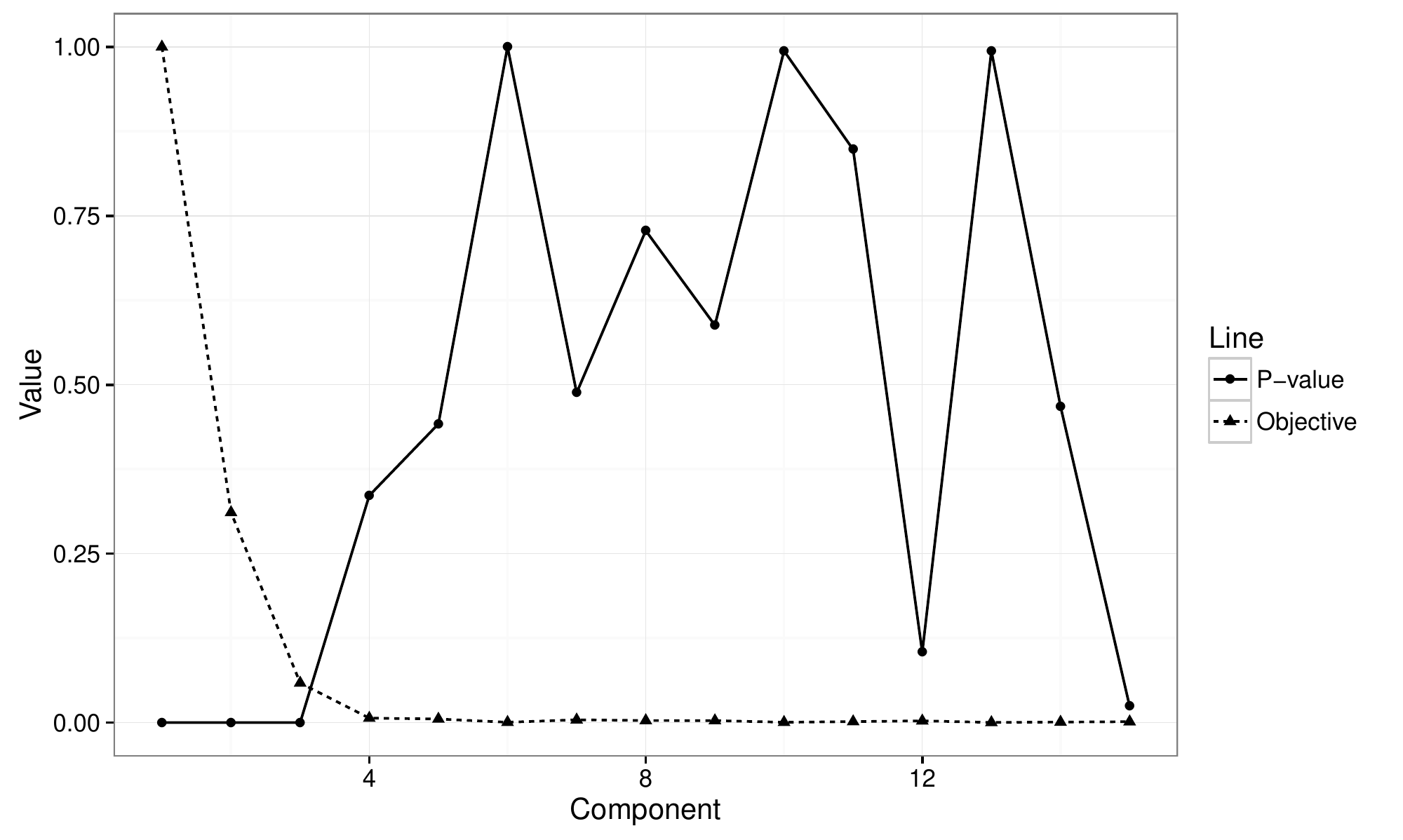}
    \caption{The resulting $p$-values and objective function values of the sequential hypothesis test for estimating the true dimension $d$.}
    \label{fig:image_d}
\end{figure}

To estimate the true number of images we used Algorithm \ref{algo_estd} with the number of normal samples $N = 500$ and the weighting $\alpha = 0.8$. The Figure \ref{fig:image_d} then shows both the resulting objective function values of each step (scaled to $(0, 1)$) and the generated $p$-values for the null hypothesis that the currently estimated component along with the remaining ones are just Gaussian noise. Both lines correctly indicate that $d=3$ is the true dimension; the objective function values are very small beginning from the fourth estimated component and at significance level $\beta = 0.05$ only the last of them differs significantly from the ones estimated from pure noise. However, the high number of tests done makes this anomaly easily attributed to randomness.

\section{DISCUSSION} \label{sec:conc}

In this paper we introduced a blind source separation model that combines the traditional ICA and NGCA models assuming that the observed random vectors are linear transformations of latent vectors having $d$ non-Gaussian signals and $p-d$ channels of Gaussian noise. To estimate the signals we proposed two different projection pursuit methods; the first one estimates the signals one-by-one and the second one simultaneously. Our projection index of choice was a convex combination of squared third and fourth cumulants yielding the advantage of finding also some signals that could be treated as noise when using only one of the cumulants.
Naturally also other \textit{nonlinearity functions} can be used in practice by replacing the objective function $G$ and its gradient $\textbf{T}$ appropriately. This can free us of making any assumptions on the existence of higher moments but, however, then the validity of Theorems \ref{theo:sep_ineq} and \ref{theo:sym_ineq} should be verified for the alternative function.

For both described methods we first gave a precise definition and a proof that the approach actually finds the correct solution. Next, an algorithm along with the affine equivariance property, estimating equations and the asymptotic behavior were given and discussed for both approaches.

Using extensive simulations, asymptotic computations and a real data example we investigated the methods' abilities to separate the signals both from the noise and from each other. The results showed that both approaches are asymptotically equally good in separating the noise from the signals but the simultaneous estimation showed slight superiority in separating individual signals coming from particular parametric families from each other. Both simulations and asymptotic theory also indicated that for data with low signal-to-noise ratio the choice of the method is largely irrelevant, the relative difference between the methods diminishing as the number of noise is increased. Furthermore, we proposed an experimental method for using the deflation-based approach to estimate the true signal dimension and showed its usefulness with both a simulation and a real data example. 

Prospective work includes: approaching the estimation problems from the viewpoint of joint cumulants, specifically using FOBI and JADE as the basis\if0\blind{ (this approach constitutes the second part of \cite{virta2015joint}, the unpublished manuscript on which the current treatise is partially based on)}\fi, giving a proper theoretical treatise of the estimation of the unknown dimension $d$ for both the deflation-based and symmetric approaches, and finally, the implementation of the procedures e.g. in the form of an R-package.

\appendix

\section{APPENDIX}

\subsection{Some notation}

We begin by providing notation used in the proofs of the asymptotical behaviors. The limiting distributions of our unmixing matrix estimates depend on the joint limiting distributions of
%The following four estimators are the building blocks of the asymptotic expressions and the covariances of all possible pairwise combinations of them are further listed in Table \ref{tab:cov}.
\begin{alignat*}{3}
\sqrt{n} \hat{s}_{kl} &= \dfrac{1}{\sqrt{n}}\sum_{i=1}^n z_{ik} z_{il}, \qquad & \\
\sqrt{n} \hat{r}_{kl} &= \dfrac{1}{\sqrt{n}}\sum_{i=1}^n (z_{ik}^2 - 1) z_{il}, \qquad & \sqrt{n} \hat{r}_{mkl} &= \frac{1}{\sqrt{n}} \sum_{i=1}^n z_{im} z_{ik} z_{il}, \\
\sqrt{n} \hat{q}_{kl} &= \frac{1}{\sqrt{n}} \sum_{i=1}^n (z_{ik}^3 - \gamma_k) z_{il} \qquad \mbox{and} \qquad & \sqrt{n} \hat{q}_{mkl} &= \frac{1}{\sqrt{n}} \sum_{i=1}^n z_{im}^2 z_{ik} z_{il}.
\end{alignat*}
The central limit theorem can be used to prove the joint limiting multivariate normality of these statistics with the variances and covariances as listed in Table \ref{tab:cov}. Finally, by $\textbf{e}_k$, $k=1,\ldots,p$, we denote the standard basis vectors of $\mathbb{R}^p$, by $\textbf{E}^{kl} := \textbf{e}_k \textbf{e}_l^T$ the matrix with one as the element $(k,l)$ and rest of the entries zero, and by $\tilde{\textbf{z}}_i := \textbf{z}_i - \bar{\textbf{z}}$ the centered identity-mixed observations.

%\sqrt{n} \hat{t}_{kl} &= \dfrac{1}{\sqrt{n}}\sum_{i=1}^n z_{ik}^2 z_{il}

\begin{table}[t]
\centering\caption{Covariances of the column and row entries, for $k \neq l \neq m \neq m'$.}
\label{tab:cov}
\begin{center}
\begin{tabular}{ l || c | c | c | c | c | c | c }
 & $\sqrt{n}\hat{q}_{kl}$ & $\sqrt{n}\hat{q}_{lk}$ & $\sqrt{n}\hat{r}_{kl}$ & $\sqrt{n}\hat{r}_{lk}$ & $\sqrt{n} \hat{q}_{m'kl}$ & $\sqrt{n} \hat{r}_{mkl}$ & $\sqrt{n}\hat{s}_{kl}$ \\ \hline \hline
$\sqrt{n}\hat{q}_{kl}$ & $\omega_k$ & $\beta_k \beta_l$ & $\eta_k$ & $\beta_k \gamma_l$ & $\beta_k$ & 0 & $\beta_k$ \\
$\sqrt{n}\hat{q}_{lk}$ & $-$ & $\omega_l$ & $\beta_l \gamma_k$ & $\eta_l$ & $\beta_l$ & 0 & $\beta_l$ \\
%$\hat{t}_{kl}$ &  &  & $\beta_k$ & $\gamma_k \gamma_l$ & & & 0 & 1 & $\gamma_k$ \\ \hline
%$\hat{t}_{lk}$ &  &  &  & $\beta_l$ & & & 1 & 0 & $\gamma_l$ \\ \hline
$\sqrt{n}\hat{r}_{kl}$ & $-$ & $-$ & $\nu_k$ & $\gamma_k \gamma_l$ & $\gamma_k$ & 0 & $\gamma_k$ \\
$\sqrt{n}\hat{r}_{lk}$ & $-$ & $-$ & $-$ & $\nu_l$ & $\gamma_l$ & 0 & $\gamma_l$ \\
%$\bar{z}_k$ &  &  &  & & & & 1 & 0 & 0 \\ \hline
%$\bar{z}_k$ &  &  &  & & & &  & 1 & 0 \\ \hline
$\sqrt{n}\hat{q}_{m'kl}$ & $-$ & $-$ & $-$ & $-$ & $\beta_m$ & 0 & 1 \\
$\sqrt{n}\hat{r}_{mkl}$ & $-$ & $-$ & $-$ & $-$ & $-$ & 1 & 0  \\
$\sqrt{n}\hat{s}_{kl}$ & $-$ & $-$ & $-$ & $-$ & $-$ & $-$ & 1 \\
\end{tabular}
\end{center}
\end{table}

\subsection{Proofs of Section \ref{sec:mod_ssf}}

\begin{proof}[Proof of Corollary \ref{cor:identify}]
The claim is equivalent to saying that if we have $\boldsymbol{\Omega}^*_1 := \boldsymbol{\Omega}_1 \textbf{A}^{-1}$ and $\textbf{s}^* := \textbf{A} \textbf{s}$ where $\textbf{A} = (a_{ij}) \in \mathbb{R}^{d \times d}$ is invertible and $\textbf{s}^*$ has standardized independent components, then the matrix $\textbf{A}$ must be of form $\textbf{D} \textbf{P}$ where $\textbf{D} \in \mathcal{D}$ and $\textbf{P} \in \mathcal{P}$.

We will first show that %no element of $\textbf{s}^*$ can be Gaussian as otherwise by Cram{\'e}r's theorem at least one element of $\textbf{s}$ would also be Gaussian, a contradiction.
no column of $\textbf{A}$ can have more than one non-zero element. To see this, take any two distinct components $s_i^*$ and $s_j^*$, $i \neq j$, of $\textbf{s}^*$. As $s_i^*$ and $s_j^*$ are independent under our assumptions so are clearly the random variables
\[\sum_{k:a_{ik} a_{jk} \neq 0} a_{ik} s_k \quad \mbox{and} \quad \sum_{k:a_{ik} a_{jk} \neq 0} a_{jk} s_k,\]
obtained by ignoring the independent components not present in both sums. Now, by Skitovich-Darmois theorem all summands in both sums must be Gaussian but our assumption on their non-Gaussianity contradicts this.  The only way both cases can hold simultaneously is when we are actually summing over an empty set, and thus $\{ k:a_{ik} a_{jk} \neq 0 \} = \emptyset$, $\forall i \neq j$. Therefore, each column of $\textbf{A}$ can have at most one non-zero element.

Secondly, as $\textbf{A}$ is assumed invertible all of its rows must have at least one non-zero element, and $\textbf{A}$ being a square matrix the only way both this and the previous condition can be satisfied is when $\textbf{A} \in \mathcal{C}$. Furthermore, as the components of $\textbf{s}^*$ are assumed to be standardized, the non-zero elements of $\textbf{A}$ have to be $\pm 1$ and thus $\textbf{A} = \textbf{D} \textbf{P}$.
\end{proof}

\subsection{Proofs of Section \ref{sec:sep}}

\begin{proof}[Proof of Theorem \ref{theo:sep_ineq}]
Theorem \ref{theo:icm_rotate} says that $\textbf{s} = \textbf{U} \textbf{x}_{st}$ but actually a stronger statement can be made: Theorem 2 in \cite{miettinen2014fourth} implies that the whole vector of independent components satisfies $\textbf{z} = \textbf{V} \textbf{x}_{st}$ for some orthogonal matrix $\textbf{V} = (\textbf{v}_1, \ldots, \textbf{v}_p)^T$ where naturally $(\textbf{v}_1, \ldots, \textbf{v}_d)^T = \textbf{U}$.  Using this it is easy to check that the claim of the theorem is equivalent to showing that for any fixed $k \geq 1$ we have 
\[
\alpha \gamma^2(\textbf{u}^T \textbf{z}) + (1 - \alpha) \kappa^2(\textbf{u}^T \textbf{z}) \leq \alpha \gamma_k^2 + (1-\alpha) \kappa_k^2,
\]
for all $\textbf{u} \in \mathbb{R}^p$ satisfying $\textbf{u}^T \textbf{u} = 1$ and $\textbf{u}^T \textbf{e}_l = 0$ for all standard basis vectors $\textbf{e}_l$, $l=1,\ldots,k-1$.

To prove this equivalent claim note first that due to the additivity and homogeneity of cumulants the following two identities hold under the model.
\[\gamma(\textbf{u}^T\textbf{z}) = \sum_{l=1}^p u_l^3\gamma_l \quad \text{and} \quad \kappa(\textbf{u}^T\textbf{z}) = \sum_{l=1}^p u_l^4\kappa_l. \]
The above assumptions on orthogonality further imply that $u_l = 0$, $l = 1,\ldots,k-1$. Then, by using the Cauchy-Schwarz inequality and the fact that $u_l^2 \leq 1$, $l = 1,\ldots,p$, we have
\begin{align*}
\alpha \gamma^2(\textbf{u}^T\textbf{z}) + (1-\alpha) \kappa^2(\textbf{u}^T\textbf{z}) \,
\leq & \alpha \left( \sum_{l=1}^p u_l^2 \right) \sum_{l=1}^p u_l^4 \gamma_l^2 + (1-\alpha) \left( \sum_{l=1}^p u_l^2 \right) \sum_{l=1}^p u_l^6 \kappa_l^2 \\
= & \alpha \sum_{l=1}^p u_l^4 \gamma_l^2 + (1-\alpha)\sum_{l=1}^p u_l^6 \kappa_l^2 \\
\leq & \sum_{l=k}^p u_l^2 \left( \alpha \gamma_l^2 + (1-\alpha) \kappa_l^2 \right) \\
\leq & \sum_{l=k}^p u_l^2 \left(\alpha \gamma_k^2 + (1-\alpha) \kappa_k^2 \right) \\
\leq & \alpha \gamma_k^2 + (1-\alpha) \kappa_k^2,
\end{align*}
where the second-to-last step holds because we without loss of generality assumed that the signals be ordered according to the values $\alpha \gamma^2 + (1-\alpha) \kappa^2$.
\end{proof}

\begin{proof}[The derivation of Algorithm \ref{algo:sep}]
The Lagrangian of the maximization problem involving $\textbf{u}_k$ has the form
\begin{align*}
 L(\textbf{u}_k, \boldsymbol{\lambda}_k) &= \alpha \left( E\left[(\textbf{u}_k^T \textbf{x}_{st})^3 \right] \right)^2 + (1 - \alpha) \left( E\left[(\textbf{u}_k^T \textbf{x}_{st})^4 \right] - 3 \right)^2 \\
 &- \sum_{j=1}^{k-1} \lambda_{kj} \textbf{u}_j^T \textbf{u}_k - \lambda_{kk} (\textbf{u}_k^T \textbf{u}_k - 1).
\end{align*}
First differentiating w.r.t. $\textbf{u}_k$ and the Lagrangian multipliers and then solving for the Lagrangian multipliers (by multiplying by $\textbf{u}_j^T$, $j=1,\ldots,k$, from the left) and substituting them back in yields the following estimating equation for the $k$th row $\textbf{u}_k$.
\[
\left(\textbf{I}_p- \sum_{j=1}^k \textbf{u}_j \textbf{u}_j^T \right)  \textbf{T}_\alpha(\textbf{u}_k)=\textbf{0}.
\]
This implies that the solution satisfies $\textbf{u}_k \propto \left(\textbf{I}_p - \sum_{j=1}^{k-1} \textbf{u}_j \textbf{u}_j^T \right) \textbf{T}(\textbf{u}_k)$ yielding Algorithm \ref{algo:sep}.
\end{proof}

To prove Theorem \ref{theo:sep_asymp} we first present the following lemma.

\begin{aplemma}\label{lem:sep_asymp}
Let $\textbf{x}_1,\ldots,\textbf{x}_n$ be a random sample from the independent component model in \eqref{eq:icm_icmodel} with $\boldsymbol{\Omega} = \textbf{I}_p$.
Assume that the eighth moments exist and that $\min_{1\le j \le d} \{\alpha\gamma_j^2+(1-\alpha)\kappa_j^2 \}> 0$. Then there exists a sequence of solutions $\hat{\textbf{W}}$ such that $\hat{\textbf{W}}\rightarrow_P (\textbf{I}_d, \textbf{0})$ and
\begin{align*}
\sqrt{n} \hat{w}_{kl} &= -\sqrt{n} \hat{w}_{lk} - \sqrt{n} \hat{s}_{kl} + o_P(1), & l < k, \\
\sqrt{n} (\hat{w}_{kk} - 1) &= -\frac{1}{2} \sqrt{n} (\hat{s}_{kk} - 1) + o_P(1),  \\
\sqrt{n} \hat{w}_{kl} &= \frac{3 \alpha \sqrt{n} \hat{\psi}_{1kl} + 4 (1 - \alpha) \sqrt{n} \hat{\psi}_{2kl}}{3 \alpha \gamma_k^2 + 4 (1 - \alpha) \kappa_k^2} + o_P(1), & l > k,
\end{align*}
where $\hat{\psi}_{1kl} = \gamma_k \hat{r}_{kl} - \gamma_k^2 \hat{s}_{kl}$ and $\hat{\psi}_{2kl} = \kappa_k \hat{q}_{kl} - \kappa_k \beta_k \hat{s}_{kl}$.
\end{aplemma}

\begin{proof}[Proof of Lemma \ref{lem:sep_asymp}]

The consistency of the estimator can be proven similarly as in \citet{miettinen2014fourth}: notice first that the population and sample objective functions, which now have the general forms
\[D(\textbf{u}) = \sum_{j=1}^J w_j \left( E\left[G_j(\textbf{u}^T \textbf{z})\right] \right)^2 \quad \text{and} \quad D_n(\textbf{u}) = \sum_{j=1}^J w_j \left( \frac{1}{n} \sum_{i=1}^n G_j(\textbf{u}^T \textbf{x}_{st,i}) \right)^2,\]
are continuous for our choices of functions $G_j$. Then, the  uniform law of large numbers in conjunction with the compactness of the unit sphere guarantees that $\text{sup}_{\textbf{u}^T \textbf{u} = 1} | D(\textbf{u}) - D_n(\textbf{u}) | \rightarrow_P 0$. Thus for each $\hat{\textbf{u}}_k$, $k=1,\ldots,d$, we can choose a sequence of solutions converging to the population maximums, $\hat{\textbf{u}}_k \rightarrow_P \textbf{e}_k$, implying then that $\hat{\textbf{w}}_k = \hat{\textbf{S}}{}^{-1/2} \hat{\textbf{u}}_k \rightarrow_P \textbf{e}_k$.

For the asymptotic behavior of $\hat{\textbf{W}} = (\hat{\textbf{w}}_1, \ldots, \hat{\textbf{w}}_d)$ we require in the current proof and the proof of Lemma \ref{lem:sym_asymp} the following estimators which assume a given estimate $\hat{\textbf{W}}$.
\begin{alignat*}{4}
& \hat{h}_{3k} = \frac{1}{n} \sum_{i=1}^n (\hat{\textbf{w}}_k^T \tilde{\textbf{z}}_i)^3, \qquad & \hat{\textbf{t}}_{3k} &= \frac{1}{n} \sum_{i=1}^n (\hat{\textbf{w}}_k^T \tilde{\textbf{z}}_i)^2 \tilde{\textbf{z}}_i, \\
& \hat{h}_{4k} = \frac{1}{n} \sum_{i=1}^n (\hat{\textbf{w}}_k^T \tilde{\textbf{z}}_i)^4 - 3 \qquad \mbox{and} \qquad & \hat{\textbf{t}}_{4k} &= \frac{1}{n} \sum_{i=1}^n (\hat{\textbf{w}}_k^T \tilde{\textbf{z}}_i)^3 \tilde{\textbf{z}}_i,
\end{alignat*}
satisfying $\hat{h}_{3k} \rightarrow_P \gamma_k, \hat{h}_{4k} \rightarrow_P \kappa_k, \hat{\textbf{t}}_{3k} \rightarrow_P \gamma_k \textbf{e}_k$ and $\hat{\textbf{t}}_{4k} \rightarrow_P \beta_k \textbf{e}_k$. In terms of $\hat{\textbf{W}} = \hat{\textbf{U}} \hat{\textbf{S}}{}^{-1/2}$ the estimating equations for $\hat{\textbf{w}}_k$ have the form
\begin{align} \label{eq:sep_sample}
\hat{\textbf{T}}_{\alpha k} = \hat{\textbf{S}} \left( \sum_{j=1}^k \hat{\textbf{w}}_j \hat{\textbf{w}}_j^T \right) \hat{\textbf{T}}_{\alpha k},
\end{align}
where $\hat{\textbf{T}}_{\alpha k} := 3 \alpha \hat{h}_{3k} \hat{\textbf{t}}_{3k} + 4 (1 - \alpha) \hat{h}_{4k}  \hat{\textbf{t}}_{4k} \rightarrow_P \lambda_k \textbf{e}_k$ and $\lambda_k := 3 \alpha \gamma_k^2 + 4 (1 - \alpha) \kappa_k \beta_k$. Then, using the equation after (4) in \cite{nordhausen2011deflation} we get the identity
\begin{align*}
\textbf{J}_k \sqrt{n} (\hat{\textbf{T}}_{\alpha k} - \lambda_k \textbf{e}_k) =& \lambda_k [ \sqrt{n} (\hat{\textbf{S}} - \textbf{I}_p) \textbf{e}_k + \sum_{j=1}^k \textbf{E}^{jk} \sqrt{n} ( \hat{\textbf{w}}_j - \textbf{e}_j) \numberthis \label{eq:app_nord5} \\
+& \sqrt{n}(\hat{\textbf{w}}_k - \textbf{e}_k) ] + o_P(1),
\end{align*}
where $\textbf{J}_k = \sum_{j>k} \textbf{E}^{jj}$. Next, using (3) from \cite{nordhausen2011deflation} separately for $\hat{\textbf{t}}_{3k}$ and $\hat{\textbf{t}}_{4k}$ gives the following two identities.
\begin{align} \label{eq:app_t3k}
\sqrt{n} (\hat{\textbf{t}}_{3k} - \gamma_k \textbf{e}_k) &= \sqrt{n} \hat{\textbf{r}}_k - 2 \textbf{E}^{kk} \sqrt{n} \bar{\textbf{z}} + 2 \gamma_k \textbf{E}^{kk} \sqrt{n} (\hat{\textbf{w}}_k - \textbf{e}_k) + o_P(1),
\end{align}
and
\begin{align*} 
\sqrt{n} (\hat{\textbf{t}}_{4k} - \beta_k \textbf{e}_k) &= \sqrt{n} \hat{\textbf{q}}_k - 3 \gamma_k \textbf{E}^{kk} \sqrt{n} \bar{\textbf{z}} \numberthis\label{eq:app_t4k}  \\
&+ 3 (\textbf{I}_p + (\beta_k - 1) \textbf{E}^{kk}) \sqrt{n} (\hat{\textbf{w}}_k - \textbf{e}_k) + o_P(1),
\end{align*}
where $\hat{\textbf{r}}_k = (1/n) \sum_{i=1}^n (z_{ik}^2 - 1) \textbf{z}_i$ and $\hat{\textbf{q}}_k = (1/n) \sum_{i=1}^n (z_{ik}^3 - \gamma_k) \textbf{z}_i$.

Using \eqref{eq:app_t3k} and \eqref{eq:app_t4k} together with the identity obtainable with Slutsky's theorem that $\sqrt{n} (\hat{h}_{3k} \hat{\textbf{t}}_{3k} - \gamma_k^2 \textbf{e}_k) = \gamma_k \sqrt{n} ( \hat{\textbf{t}}_{3k} - \gamma_k \textbf{e}_k) + \gamma_k \sqrt{n} (\hat{h}_{3k} - \gamma_k) \textbf{e}_k + o_P(1)$ (and the analogy for $\hat{\textbf{t}}_{4k}$) we get an alternative expression for $\sqrt{n} (\hat{\textbf{T}}_{\alpha k} - \lambda_k \textbf{e}_k)$ which can be substituted into \eqref{eq:app_nord5}. Inspecting the result element-wise ($l=1,\ldots,p$) then yields the following three equations from which the result follows by noting that $\hat{\textbf{S}}$ is symmetric.
\begin{align*}
0 &= \sqrt{n} \hat{s}_{lk} + \sqrt{n} \hat{w}_{lk} + \sqrt{n} \hat{w}_{kl} + o_P(1), \quad &l < k, \\
0 &= \sqrt{n} (\hat{s}_{kk} - 1) + 2 \sqrt{n} (\hat{w}_{kk} - 1) + o_P(1), \quad &l = k,
\end{align*}
and
\begin{align*}
\quad & 3 \alpha \gamma_k \sqrt{n} \hat{r}_{kl} + 4 (1 - \alpha) \kappa_k ( \sqrt{n} \hat{q}_{kl} + 3 \sqrt{n} \hat{w}_{kl})\\
= \quad & (3 \alpha \gamma_k^2 + 4 (1 - \alpha) \kappa_k \beta_k)(\sqrt{n} \hat{s}_{lk} + \sqrt{n} \hat{w}_{kl}) + o_P(1), \quad l > k.
\end{align*}
\end{proof}

\begin{proof}[Proof of Theorem \ref{theo:sep_asymp}]
The expressions of Theorem \ref{theo:sep_asymp} follow straightforwardly by computing the variances of the expressions given in Lemma \ref{lem:sep_asymp} using Table \ref{tab:cov} given in the beginning of this supplementary material.
\end{proof}

\subsection{Proofs of Section \ref{sec:sym}}

For the proof of Theorem \ref{theo:sym_ineq} we first present and prove the following lemma.

\begin{aplemma}\label{lem:app_ineq}
Let the matrix $\textbf{V} \in \mathbb{R}^{d \times p}$ have orthonormal rows, $\textbf{b} \in \mathbb{R}^p$ and $r \in \mathbb{N}, \, r \geq 2$. Then
\[\sum_{k=1}^d \left( \sum_{l=1}^p v_{kl}^r  b_l \right)^2 \leq \sum_{l=1}^p b_l^2. \]
\end{aplemma}
\begin{proof}[Proof of Lemma \ref{lem:app_ineq}]
Utilize first the Cauchy-Schwarz inequality:
\[ \sum_{k=1}^d \left( \sum_{l=1}^p v_{kl}^r  b_l \right)^2 = \sum_{k=1}^d \left( \sum_{l=1}^p (v_{kl}) (v_{kl}^{r-1} b_l) \right)^2 \leq \sum_{k=1}^d \sum_{l=1}^p v_{kl}^{2r-2} b_l^2. \]
Then observing that $v_{kl}^{2r-2} = v_{kl}^2 v_{kl}^{2r-4} \leq v_{kl}^2$ and $\sum_{k=1}^d v_{kl}^2 \leq 1$ gives the desired result.
\end{proof}

\begin{proof}[Proof of Theorem \ref{theo:sym_ineq}]
As in the proof of Theorem \ref{theo:sep_ineq} we can write the claim in the equivalent form:
\[
\sum_{k=1}^d \left( \alpha \gamma^2(\textbf{v}_k^T \textbf{z}) + (1 - \alpha) \kappa^2(\textbf{v}_k^T \textbf{z}) \right) \leq \sum_{k=1}^d \left( \alpha \gamma_k^2 + (1-\alpha) \kappa_k^2 \right),
\]
for all $\textbf{V} = (\textbf{v}_1,\ldots,\textbf{v}_d)^T \in \mathcal{U}^{d \times p}$ with orthonormal rows (note that the matrix $\textbf{V}$ now differs from the $\textbf{V}$ in the original formulation).

The above formulation is then easily proved by first expanding the left-hand side under the assumptions of our model in \eqref{eq:icm_icmodel} to yield (see also the proof of Theorem \ref{theo:sep_ineq})
\[
\alpha \sum_{k=1}^d \left( \sum_{l=1}^p v_{kl}^3  \gamma_l \right)^2 + (1 - \alpha) \sum_{k=1}^d \left( \sum_{l=1}^p v_{kl}^4  \kappa_l \right)^2.
\]
Then applying Lemma \ref{lem:app_ineq} and substituting $\gamma_k = \kappa_k = 0$ for the noise components, $k = d+1,\ldots,p$, gives the desired result.
\end{proof}

\begin{proof}[The derivation of Algorithm \ref{algo:sym1}]
The Lagrangian of the optimization problem has the form
\begin{align*}
 L(\textbf{U}, \boldsymbol{\Lambda}) &= \alpha \sum_{k=1}^d \left( E\left[(\textbf{u}_k^T \textbf{x}_{st})^3 \right] \right)^2 + (1 - \alpha) \sum_{k=1}^d \left( E\left[(\textbf{u}_k^T \textbf{x}_{st})^4 \right] - 3 \right)^2 \\
 &- \sum_{k=1}^{d-1} \sum_{l=k+1}^d \lambda_{kl} \textbf{u}_k^T \textbf{u}_l - \sum_{k=1}^d \lambda_{kk} (\textbf{u}_k^T \textbf{u}_k - 1).
\end{align*}
First differentiating w.r.t. $\textbf{U}$ and the Lagrangian multipliers in $\boldsymbol{\Lambda}$ and solving for the multipliers as in the derivation of Algorithm \ref{algo:sep} we notice that the multipliers have two solutions that must be equal, thus yielding the estimating equations
\[
\textbf{u}_l^T \textbf{T}_\alpha(\textbf{u}_k) =\textbf{u}_k^T \textbf{T}_\alpha(\textbf{u}_l), \quad \forall k,l=1,\ldots,d,
\]
and $\textbf{U}\textbf{U}^T = \textbf{I}_d$. Writing $\textbf{T}_\alpha(\textbf{U}) := (\textbf{T}_\alpha(\textbf{u}_1), \ldots, \textbf{T}_\alpha(\textbf{u}_d))^T \in \mathbb{R}^{d \times p}$ we then get, as in \cite{miettinen2014fourth}, the following compact matrix form
\begin{equation}\label{eq:ee1}
\textbf{U} \textbf{T}_\alpha(\textbf{U})^T = \textbf{T}_\alpha(\textbf{U}) \textbf{U}^T \ \ \mbox{and} \ \  \textbf{U}\textbf{U}^T = \textbf{I}_d.
\end{equation}
Alternatively, substituting the solved multipliers back in as in the derivation of Algorithm \ref{algo:sep} we get a third set of estimating equations:
\begin{equation}\label{eq:ee2}
\textbf{T}_\alpha(\textbf{U})^T = \textbf{U}^T \textbf{U} \textbf{T}_\alpha(\textbf{U})^T.
\end{equation}
Assume then a fixed $\textbf{U}$ that satisfies the estimating equations in \eqref{eq:ee1} and \eqref{eq:ee2}. The matrix $\textbf{U} \textbf{T}_\alpha(\textbf{U})^T$ being symmetric by \eqref{eq:ee2} admits the eigendecomposition $\textbf{U} \textbf{T}_\alpha(\textbf{U})^T = \textbf{V} \textbf{D} \textbf{V}^T$ where $\textbf{V} \in \mathcal{U}$. Plugging then the decomposition in \eqref{eq:ee2} yields $\textbf{T}_\alpha(\textbf{U})^T = \textbf{U}^T \textbf{V} \textbf{D} \textbf{V}^T$, the singular value decomposition of $\textbf{T}_\alpha(\textbf{U})^T$. Assume then that the matrix $\textbf{T}_\alpha(\textbf{U})^T$ is of full column rank which implies that $\textbf{D}$ is positive definite and thus invertible. 

Next, substituting the expression for $\textbf{T}_\alpha(\textbf{U})^T$  in to the formula iterated in Algorithm \ref{algo:sym1}, $(\textbf{T}_\alpha(\textbf{U}) \textbf{T}_\alpha(\textbf{U})^T)^{-1/2} \textbf{T}_\alpha(\textbf{U})$, shows that the solution $\textbf{U}$ must satisfy
\begin{align*}%\label{eq:ee3}
(\textbf{T}_\alpha(\textbf{U}) \textbf{T}_\alpha(\textbf{U})^T)^{-1/2} \textbf{T}_\alpha(\textbf{U}) = \textbf{U}.
\end{align*}
That the global maximum $\textbf{U}$ we are trying to estimate indeed yields $\textbf{T}_\alpha(\textbf{U})^T$ with full column rank is easily checked by plugging-in, thus yielding the proposed algorithm.

%Moreover, it is easily checked that any $\textbf{U}$ satisfying the obtained estimating equations in \eqref{eq:ee3} also satisfies the original estimating equations in \eqref{eq:ee1} and \eqref{eq:ee2}.

%Thus the set $\mathcal{S}_2$ of matrices satisfying \eqref{eq:ee3} is a subset of the set $\mathcal{S}_1$ of matrices satisfying both \eqref{eq:ee1} and \eqref{eq:ee2}, $\mathcal{S}_2 \subset \mathcal{S}_1$. Furthermore, $\mathcal{S}_2$ contains those matrices $\textbf{U}$ of $\mathcal{S}_1$ for which $\textbf{T}_\alpha(\textbf{U})^T$ is of full column rank and it is easy to see by direct plugging-in that the  (that we are trying to estimate) is such a matrix, yielding the proposed algorithm. 

\end{proof}

For the proof of Theorem \ref{theo:sym_asymp} we need the following lemma.

\begin{aplemma}\label{lem:sym_asymp}
Let $\textbf{x}_1,\ldots,\textbf{x}_n$ be a random sample  from the independent component model in \eqref{eq:icm_icmodel} with $\boldsymbol{\Omega}=\textbf{I}_p$.
Assume that the eighth moments exist and that $\min_{1\le j \le d} \{\alpha\gamma_j^2+(1-\alpha)\kappa_j^2 \}> 0$.
Then there exists a sequence of solutions $\hat{\textbf{W}}$ such that $\hat{\textbf{W}} \rightarrow_P (\textbf{I}_d, \textbf{0})$ and
\begin{align*}
\sqrt{n} \hat{w}_{kl} &= \frac{3 \alpha \sqrt{n} \hat{\xi}_{1kl} + 4 (1 - \alpha)  \sqrt{n} \hat{\xi}_{2kl}}{3 \alpha (\gamma_k^2 + \gamma_l^2) + 4 (1 - \alpha) (\kappa_k^2 + \kappa_l^2)} + o_P(1), &\quad l \leq d, l \neq k, \\
\sqrt{n} (\hat{w}_{kk} - 1) &= -\frac{1}{2} \sqrt{n} (\hat{s}_{kk} - 1) + o_P(1),  \\
\sqrt{n} \hat{w}_{kl} &= \frac{3 \alpha \sqrt{n} \hat{\psi}_{1kl} + 4 (1 - \alpha) \sqrt{n} \hat{\psi}_{2kl}}{3 \alpha \gamma_k^2 + 4 (1 - \alpha) \kappa_k^2} + o_P(1), & l > d,
\end{align*}
where $\hat{\xi}_{1kl} = \gamma_k \hat{r}_{kl} - \gamma_l \hat{r}_{lk} - \gamma_k^2 \hat{s}_{kl}$, $\hat{\xi}_{2kl} = \kappa_k \hat{q}_{kl} - \kappa_l \hat{q}_{lk} - (\kappa_k \beta_k - 3\kappa_l) \hat{s}_{kl}$, $\hat{\psi}_{1kl} = \gamma_k \hat{r}_{kl} - \gamma_k^2 \hat{s}_{kl}$ and $\hat{\psi}_{2kl} = \kappa_k \hat{q}_{kl} - \kappa_k \beta_k \hat{s}_{kl}$.
\end{aplemma}

\begin{proof}[Proof of Lemma \ref{lem:sym_asymp}]
The uniform convergence in probability of the sample objective function, $D_n(\textbf{U}) = \sum_{k=1}^d D_n(\textbf{u}_k)$, to the population one, $D(\textbf{U}) = \sum_{k=1}^d D(\textbf{u}_k)$, (see the proof of Lemma \ref{lem:sep_asymp} for the definitions of $D_n(\textbf{u}_k)$ and $D(\textbf{u}_k)$) in $\mathcal{U}$ can be shown as in the proof of Theorem 6 in \cite{miettinen2014fourth}: by observing that $D_n(\textbf{U})$ and $D(\textbf{U})$ are continuous and $\mathcal{U}$ is compact and then using the uniform law of large numbers. The population objective function is then maximized by any $(\textbf{JP}, \textbf{0})$ and, specifically, there exists a sequence of solutions that satisfies $\hat{\textbf{U}} \rightarrow_P (\textbf{I}_d, \textbf{0})$ implying that $\hat{\textbf{W}} = \hat{\textbf{U}} \hat{\textbf{S}}{}^{-1/2} \rightarrow_P (\textbf{I}_d, \textbf{0})$.
%\[\left\lbrace \textbf{J} \textbf{P}
%\begin{pmatrix}
%\textbf{I}_d & \textbf{0} \\
%\textbf{0} & \textbf{V}
%\end{pmatrix}
%\, | \, \textbf{J} \in \mathcal{J}, \textbf{P} \in \mathcal{P}, %\textbf{V} \in \mathcal{U}
%\right\rbrace, \]
%where the rotation $\textbf{V}$ is necessary because of the rotational invariance of the noise subspace. First, this together with the uniform convergence implies that there exists a sequence of solutions that satisfies $\hat{\textbf{W}}_1 = \hat{\textbf{U}}_1 \hat{\textbf{S}}{}^{-1/2} \rightarrow_P \textbf{I}_{d, p}$ and second, it does not actually matter what the value of $\hat{\textbf{W}}_2 := (\textbf{0}, \hat{\textbf{V}})$ is; we can always pretend that the rotationally invariant noise vector $(z_{d+1},...,z_p)^T$ is multiplied by its inverse, $\hat{\textbf{V}}{}^T$, cancelling each other out. Thus, without loss of generality, we may assume that the sequence of solutions also satisfies $\hat{\textbf{W}} := \hat{\textbf{U}} \hat{\textbf{S}}{}^{-1/2} \rightarrow_P \textbf{I}$.

For the asymptotic behavior itself, we have to use two different sets of estimating equations, the first one in \eqref{eq:ee1} and the one in \eqref{eq:ee2}. Starting with the latter, the sample version for the $k$th column of it is exactly as \eqref{eq:sep_sample} but with $\sum_{j=1}^k \hat{\textbf{w}}_j \hat{\textbf{w}}_j{}^T$ replaced by $\sum_{j=1}^d \hat{\textbf{w}}_j \hat{\textbf{w}}_j{}^T$ for all $k$. Consequently, we have for all $k$ the identities \eqref{eq:app_nord5} with $\textbf{J}_k$ replaced by $\textbf{J}_d$ and the upper limit of the sum replaced by $d$.

Observing then the $l$th component with either $l \leq d$ or $l > d$ in the equivalents of \eqref{eq:app_nord5} gives using the techniques of the proof of Lemma \ref{lem:sep_asymp} the identities
\begin{align*}
0 &= \sqrt{n} \hat{s}_{lk} + \sqrt{n} \hat{w}_{lk} + \sqrt{n} \hat{w}_{kl} + o_P(1), \quad &l \leq d, \ l \neq k, \\
0 &= \sqrt{n} (\hat{s}_{kk} - 1) + 2 \sqrt{n} (\hat{w}_{kk} - 1) + o_P(1), \quad &l \leq d,
\end{align*} 
and
\begin{align*}
\quad & 3 \alpha \gamma_k \sqrt{n} \hat{r}_{kl} + 4 (1 - \alpha) \kappa_k ( \sqrt{n} \hat{q}_{kl} + 3 \sqrt{n} \hat{w}_{kl})\\
= \quad & (3 \alpha \gamma_k^2 + 4 (1 - \alpha) \kappa_k \beta_k)(\sqrt{n} \hat{s}_{lk} + \sqrt{n} \hat{w}_{kl}) + o_P(1), \quad l > d.
\end{align*}
The first identity of the above three gives the asymptotic behavior for the off-diagonal element $(k, l)$ in the matrix $\hat{\textbf{W}}_1$ assuming that the asymptotic behavior of the element $(l, k)$ is known; the second one gives the asymptotic behavior of the diagonal elements of $\hat{\textbf{W}}_1$ and the third one gives the asymptotic behavior of the whole $\hat{\textbf{W}}_2$. The last two already yield the second and third claims of the lemma but for the first part we still need the actual asymptotic expressions for the lower triangle of $\hat{\textbf{W}}_1$ and those are provided by the first set of estimating equations in \eqref{eq:ee1}.

The sample versions of them for $k,l = 1,\ldots,d$ are
\[3 \alpha \hat{h}_{3k} \hat{\textbf{w}}_l^T \hat{\textbf{t}}_{3k} + 4 (1 - \alpha) \hat{h}_{4k} \hat{\textbf{w}}_l^T \hat{\textbf{t}}_{4k} = 3 \alpha \hat{h}_{3l} \hat{\textbf{w}}_k^T \hat{\textbf{t}}_{3l} + 4 (1 - \alpha) \hat{h}_{4l} \hat{\textbf{w}}_k^T \hat{\textbf{t}}_{4l}, \numberthis \label{eq:sym_ee}  \]
where $\hat{h}_{3k}$, $\hat{h}_{4k}$, $\hat{\textbf{t}}_{3k}$ and $\hat{\textbf{t}}_{4k}$ are as in the proof of Lemma \ref{lem:sep_asymp}. With an approach similar to the one used in the proof of Theorem 6 in \cite{miettinen2014fourth} we have
\[\sqrt{n} \hat{h}_{3l} \hat{\textbf{w}}_k^T \hat{\textbf{t}}_{3l} = \gamma_l \sqrt{n} (\hat{\textbf{w}}_k - \textbf{e}_k)^T \gamma_l \textbf{e}_l + \gamma_l \textbf{e}_k^T \sqrt{n} (\hat{\textbf{t}}_{3l} - \gamma_l \textbf{e}_l) + o_P(1), \]
and
\[\sqrt{n} \hat{h}_{4l} \hat{\textbf{w}}_k^T \hat{\textbf{t}}_{4l}  = \kappa_l \sqrt{n} (\hat{\textbf{w}}_k - \textbf{e}_k)^T \beta_l \textbf{e}_l + \kappa_l \textbf{e}_k^T \sqrt{n} (\hat{\textbf{t}}_{4l} - \beta_l \textbf{e}_l) + o_P(1). \]

Substituting \eqref{eq:app_t3k} and \eqref{eq:app_t4k} into the above expansions and then plugging in into \eqref{eq:sym_ee} gives the identity
\begin{align*}
& 3 \alpha (\gamma_k^2 \sqrt{n} \hat{w}_{lk} + \gamma_k \sqrt{n} \hat{r}_{kl}) + 4 (1 - \alpha) (\beta_k \kappa_k \sqrt{n} \hat{w}_{lk} + \kappa_k \sqrt{n} \hat{q}_{kl} + 3 \kappa_k \sqrt{n} \hat{w}_{kl}) \\
= & 3 \alpha (\gamma_l^2 \sqrt{n} \hat{w}_{kl} + \gamma_l \sqrt{n} \hat{r}_{lk}) + 4 (1 - \alpha) (\beta_l \kappa_l \sqrt{n} \hat{w}_{kl} + \kappa_l \sqrt{n} \hat{q}_{lk} + 3 \kappa_l \sqrt{n} \hat{w}_{lk}) + o_P(1),
\end{align*}
into which the first asymptotic expression obtained earlier from the second set of estimating equations can be substituted to express $\sqrt{n}\hat{w}_{lk}$ using $\sqrt{n}\hat{w}_{kl}$, finally yielding the missing first part of the claim.
\end{proof}

\begin{proof}[Proof of Theorem \ref{theo:sym_asymp}]
The expressions of Theorem \ref{theo:sym_asymp} are obtained similarly as in the proof of Theorem \ref{theo:sep_asymp}.
\end{proof}

\subsection{Proofs of Section \ref{sec:simu}}

\begin{proof}[Proof of Theorem \ref{theo:as independence}]
Assume that in the model \eqref{eq:icm_icmodel} the latent vectors $\textbf{z}_i = (\textbf{s}_i^T, \textbf{n}_i^T)^T$ are linearly transformed as
\[ \begin{pmatrix}
\textbf{s}_i \\
\textbf{n}_i
\end{pmatrix}
\mapsto
\begin{pmatrix}
\textbf{I}_d & \textbf{0} \\
\textbf{0} & \textbf{U} 
\end{pmatrix}
\begin{pmatrix}
\textbf{s}_i \\
\textbf{n}_i
\end{pmatrix},
\]
where $\textbf{U}$ is a $(p-d) \times (p-d)$ orthogonal matrix. As $\textbf{n}_i$ has the standard multivariate normal distribution this transformation leaves the distribution of the $\textbf{z}_i$ unchanged. Denoting the above block-diagonal transformation matrix as $\textbf{K}$ it follows that the signal separation functionals $\hat{\textbf{W}}(\textbf{z}_i)$ and $\hat{\textbf{W}}(\textbf{K}\textbf{z}_i) = \hat{\textbf{W}}(\textbf{z}_i)\textbf{K}^{-1}$, where the equality follows from the affine equivariance of the functional (assume that we have ordered and changed the signs of the rows of the latter functional to match the rows of the former functional), are identically distributed. Thus, in particular, the covariance matrices of their limiting distributions are the same:
\[\mbox{AsCov}\left(\mbox{vec}(\hat{\textbf{W}}(\textbf{z}_i))\right) = \mbox{AsCov}\left(\mbox{vec}(\hat{\textbf{W}}(\textbf{z}_i)\textbf{K}^{-1})\right), \]
for all orthogonal $\textbf{U}$. Using the identity $\mbox{vec}(\textbf{ABC}) = (\textbf{C}^T \otimes \textbf{A}) \mbox{vec}(\textbf{B})$, where $\otimes$ is the Kronecker product, and denoting the left-hand side of the above equation by $\textbf{A}$ yields
\begin{align}\label{eq:ascov_comm}
\textbf{A} = \left( \begin{pmatrix}
\textbf{I}_d & \textbf{0} \\
\textbf{0} & \textbf{U} 
\end{pmatrix} \otimes \textbf{I}_d \right) \textbf{A}
\left( \begin{pmatrix}
\textbf{I}_d & \textbf{0} \\
\textbf{0} & \textbf{U}^T 
\end{pmatrix} \otimes \textbf{I}_d \right),
\end{align}
for all orthogonal $\textbf{U}$. So in particular \eqref{eq:ascov_comm} has to hold when $\textbf{U} \in \mathcal{J} \subset \mathcal{U}$ and inspecting various cases shows that this can hold only if $\textbf{A}$ is block-diagonal with $p-d+1$ blocks of sizes $d^2\times d^2$, $d\times d,\ldots,d\times d$. Furthermore, noticing that \eqref{eq:ascov_comm} must hold for all $\textbf{U} \in \mathcal{P} \subset \mathcal{U}$ and again inspecting element-wise shows that it is necessary for the final $d$ blocks to be identical.
\end{proof}

\newcommand{\matw}{\hat{\textbf{W}}}
\newcommand{\mati}{(\textbf{I}_d, \textbf{0})}

\begin{proof}[Proof of Theorem \ref{theo:proj_asymp}]
First, expanding the projection matrix as
\begin{align*}
\hat{\textbf{P}} - \mati^T \mati &= (\matw{}^T - \mati^T) (\matw \matw{}^T)^{-1} \matw \\
&+ \mati^T((\matw \matw{}^T)^{-1} - \textbf{I}_d) \matw + \mati^T(\matw - \mati),
\end{align*}
and then using Slutsky's theorem yields
\begin{align*}
\sqrt{n}(\hat{\textbf{P}} - \mati^T \mati) &= \sqrt{n}(\matw{}^T - \mati^T) \mati \\ 
&+ \mati^T \sqrt{n}((\matw \matw{}^T)^{-1} - \textbf{I}_d)\mati \\
&+ \mati^T \sqrt{n}(\matw - \mati) + o_P(1).
\end{align*}
An alternative expression for the inverse term is obtained by noticing that $\sqrt{n}((\matw \matw{}^T)^{-1}(\matw \matw{}^T) - \textbf{I}_d) = \textbf{0}$ and expanding the left-hand side in above manner and again using Slutsky's theorem. This yields
\begin{align*}
\sqrt{n}((\matw \matw{}^T)^{-1} - \textbf{I}_d) &= -\sqrt{n}(\matw - \mati) \mati^T \\ 
&- \mati \sqrt{n}(\matw{}^T - \mati^T) + o_P(1),
\end{align*}
which, when substituted into the expression for $\sqrt{n}(\hat{\textbf{P}} - \mati^T \mati)$ above along with the fact that $\matw = (\matw_1, \matw_2)$, yields, after some simplification, the desired result.

\end{proof}

\subsection{The images of the real data example in Section \ref{sec:sub}}

The three true images and the 15 mixed images obtained by mixing the true images with 12 images of independent Gaussian noise in the real data example of Section \ref{sec:sub} are shown in Figures \ref{fig:true_pics} and \ref{fig:mix_pics}, respectively.

\begin{figure}[t]
    \centering
    \includegraphics[width=1.0\textwidth]{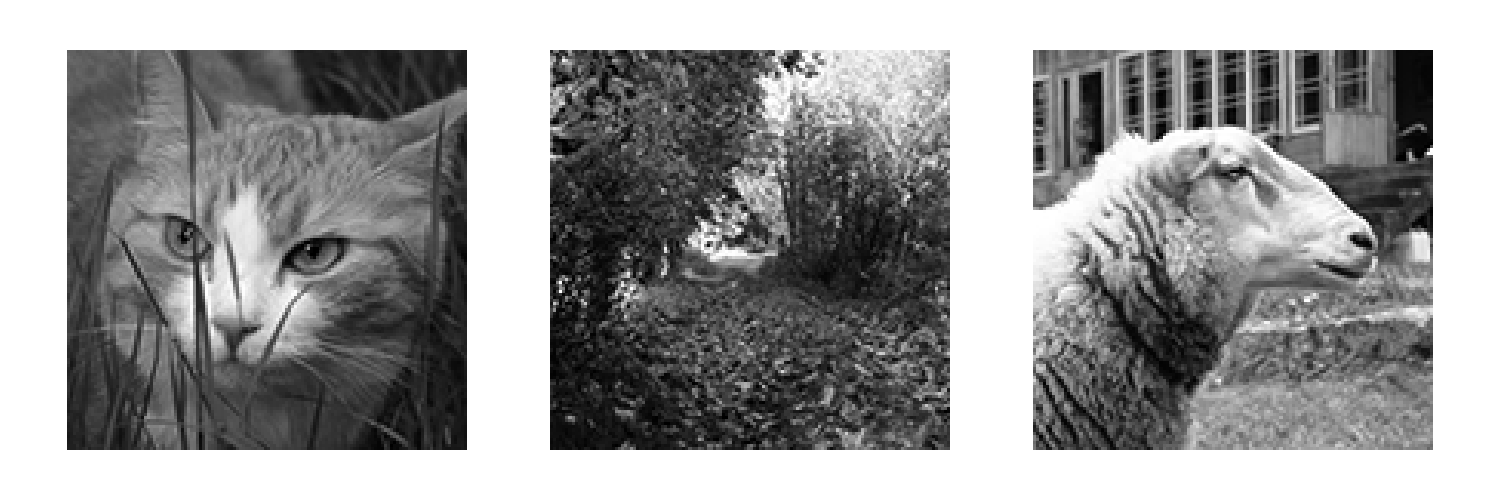}
    \caption{The three true images used in the real data example in Section \ref{sec:sub}.}
    \label{fig:true_pics}
\end{figure}

\begin{figure}[t]
    \centering
    \includegraphics[width=1.0\textwidth]{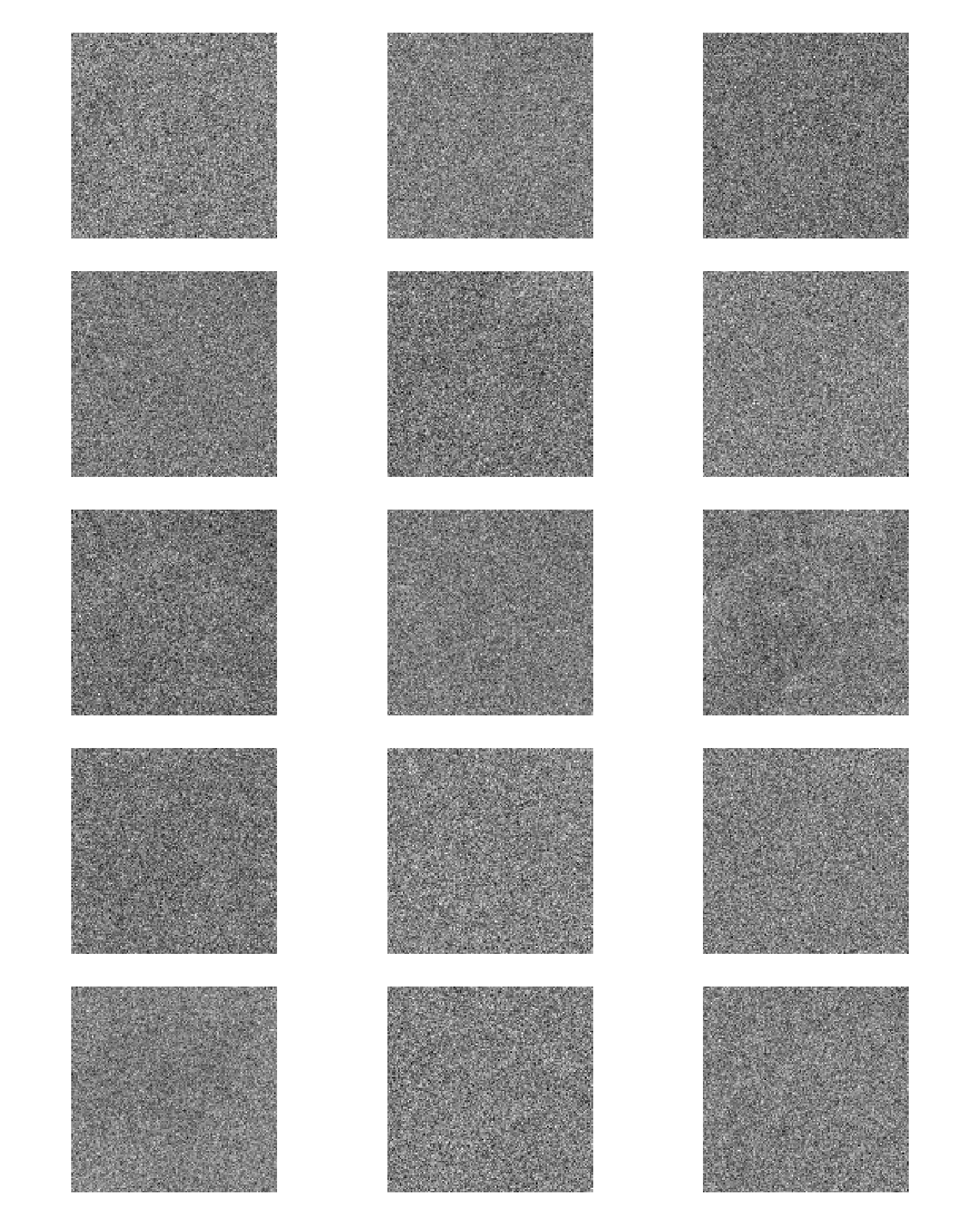}
    \caption{The 15 mixed images obtained by mixing the true images with 12 images of independent Gaussian noise in the real data example in Section \ref{sec:sub}. No discernible signal is visible in any of them.}
    \label{fig:mix_pics}
\end{figure}

\bibliographystyle{Chicago}

\bibliography{references}
\end{document}